\documentclass[12pt,reqno]{amsart}
\usepackage[margin=1.25in]{geometry}

\usepackage{amsmath,amssymb,amsfonts}
\usepackage{mathtools}

\usepackage{thm-restate}

\usepackage{tikz}
\usetikzlibrary{calc}

\usepackage{array}
\usepackage{booktabs}
\usepackage{multirow}

\usepackage[centertableaux]{ytableau}
\ytableausetup{notabloids}
\ytableausetup{centertableaux}

\usepackage{hyperref}
\hypersetup{colorlinks,allcolors=blue}

\usepackage[capitalize]{cleveref}

\usepackage[utf8x]{inputenc}

\theoremstyle{plain}
   
   \newtheorem{theorem}{Theorem}[section]

   \newtheorem{lemma}[theorem]{Lemma}
   \newtheorem{corollary}[theorem]{Corollary}
\theoremstyle{definition}
   
   \newtheorem{example}[theorem]{Example}
   
   \newtheorem{remark}[theorem]{Remark}

\numberwithin{equation}{section}

\DeclareMathOperator{\diag}{diag}
\DeclareMathOperator{\dec}{dec}

\usepackage{fonttable}
\DeclareFontFamily{U} {MnSymbolA}{}
\DeclareFontShape{U}{MnSymbolA}{m}{n}{
  <-6> MnSymbolA5
  <6-7> MnSymbolA6
  <7-8> MnSymbolA7
  <8-9> MnSymbolA8
  <9-10> MnSymbolA9
  <10-12> MnSymbolA10
  <12-> MnSymbolA12}{}
\DeclareFontShape{U}{MnSymbolA}{b}{n}{
  <-6> MnSymbolA-Bold5
  <6-7> MnSymbolA-Bold6
  <7-8> MnSymbolA-Bold7
  <8-9> MnSymbolA-Bold8
  <9-10> MnSymbolA-Bold9
  <10-12> MnSymbolA-Bold10
  <12-> MnSymbolA-Bold12}{}
\DeclareSymbolFont{MnSyA} {U} {MnSymbolA}{m}{n}
\DeclareMathSymbol{\lcurvearrownw}{\mathrel}{MnSyA}{189}
\DeclareMathSymbol{\rcurvearrowne}{\mathrel}{MnSyA}{196}

\newcommand{\ubr}[1]{\underbracket[\fontdimen8\textfont3]{#1}}

\newcommand{\rA}{\mathsf{A}}
\newcommand{\rD}{\mathsf{D}}
\newcommand{\rE}{\mathsf{E}}
\newcommand{\rJ}{\mathsf{J}}

\newcommand{\Gr}{{\rm Gr}}
\newcommand{\CC}{\mathbb{C}}

\newcommand{\eroot}[1]{\begin{matrix}#1\end{matrix}}

\newcommand{\thfrac}[3]{ 	
\begin{aligned} #1\\ \hline \\[-3\jot] #2 \\ \hline \\[-3\jot] #3 \end{aligned}
}
\begin{document}

\title{Real roots in the root system $\mathsf{T}_{2,p,q}$}
\author{Karin Baur, Jian-Rong Li, and Andrei Smolensky}
\address{Karin Baur, School of Mathematics, University of Leeds, Leeds, LS2 9JT, Currently on leave from the University of Graz, Graz, Austria.} 
\email{ka.baur@me.com}
\address{Jian-Rong Li, Faculty of Mathematics, University of Vienna, Oskar-Morgenstern-Platz 1, 1090 Vienna, Austria.} 
\email{lijr07@gmail.com}
\address{Andrei Smolensky, Department of Mathematics and Mechanics, Saint Petersburg State University, Saint Petersburg, Russia.} 
\email{andrei.smolensky@gmail.com}
\date{}

\maketitle

\begin{abstract}
Motivated by the recent advances in the categorification of the cluster structure on the coordinate rings of Grassmannians of $k$-subspaces in $n$-space, we investigate a particular construction of root systems of type $\mathsf{T}_{2,p,q}$, including the type $\mathsf{E}_n$. This construction generalizes Manin's ``hyperbolic construction'' of $\mathsf{E}_8$ and reveals a lot of otherwise hidden regularities in this family of root systems.
\end{abstract}

%\tableofcontents

\section{Introduction}
The real roots of root systems of finite, affine and hyperbolic type can be characterized in terms of the coefficients of their decomposition into the linear combination of simple roots \cite[Proposition~5.10]{Kac}. 
For the root system with a simply-laced diagram this description boils down to the following: 
the real roots are the elements of the root lattice having the same norm as the simple roots. However, 
for non-hyperbolic root systems this condition is only necessary, but not sufficient. 
There is at present no general description of real roots available for non-hyperbolic root systems. 

We investigate the root system of type $\rJ_{k,n}=\mathsf{T}_{2,k,n-k}$, $k\leqslant n$, which has the following diagram:
\begin{center}
\begin{tikzpicture}
\tikzset{dynkin-vertex/.style={draw,circle,very thick,minimum size=3mm,inner sep=0mm}}
\tikzset{dynkin-edge/.style={draw,very thick}}

\node[dynkin-vertex,label={90}:{$\alpha_1$}] (a1) at (0,0) {};
\node[dynkin-vertex,label={90}:{$\alpha_2$}] (a2) at (1,0) {};
\begin{scope}[xshift=3cm]
\node[dynkin-vertex,label={90}:{$\alpha_k$}] (ak) at (0,0) {};
\node[dynkin-vertex,label={90}:{$\alpha_{k+1}$}] (akp1) at (1,0) {};
\node[dynkin-vertex,label={180}:{$\beta$}] (b) at (0,-1) {};
\begin{scope}[xshift=3cm]
\node[dynkin-vertex,label={90}:{$\alpha_{n-1}$}] (anm1) at (0,0) {};
\end{scope}
\end{scope}
\coordinate (a2akl) at ($(a2)!.25!(ak)$);
\coordinate (a2akr) at ($(a2)!.75!(ak)$);
\coordinate (akp1anm1l) at ($(akp1)!.25!(anm1)$);
\coordinate (akp1anm1r) at ($(akp1)!.75!(anm1)$);
\draw[dynkin-edge] (a1)--(a2)--(a2akl) (a2akr)--(ak)--(akp1)--(akp1anm1l) (akp1anm1r)--(anm1) (ak)--(b);
\draw[dynkin-edge,loosely dotted,line cap=round] (a2akl)--(a2akr) (akp1anm1l)--(akp1anm1r);
\end{tikzpicture}
\end{center}

Here we denote $\beta=\alpha_n$. The root system $\rJ_{3,n}$ is usually called the $\rE_n$ root system, while $\rJ_{1,n}=\rA_n$ and $\rJ_{2,n}=\rD_n$. In general, the root system $\rJ_{k,n}$ is non-finite, non-affine, and non-hyperbolic. With this paper, we give a characterization of real roots for 
a large class of root systems. 

Such root systems appear naturally in the study of generalized Del Pezzo varieties, that is, roughly speaking, the blow-ups of $\mathbb{P}^m$ at some finite set of points, see \cite{Coble,DO}. In particular, for the case $k=3$ and its relation to the Picard lattice of Del Pezzo surfaces see \cite[Section~25]{Manin}.

Another motivation for the present paper is the study of the rigid indecomposable modules in Grassmannian cluster categories $\mathrm{CM}(B_{k,n})$, see~\cite{JKS, BBG} and cluster variables in Grassmannian cluster algebras $\CC[\Gr_{k,n}]$ \cite{Sco06}. Cluster algebras are a class of commutative rings introduced by S. Fomin and A. Zelevinsky in their series of foundational papers \cite{BFZ, FZ1, FZ2, FZ3} (the paper \cite{BFZ} is with coauthor A. Berenstein). 
Scott proved that there is a cluster algebra structure on the coordinate ring $\CC[\Gr_{k,n}]$ of the Grassmannian varieties $\Gr_{k,n}$. 
Jensen, King and Su in~\cite{JKS} showed that the category ${\rm CM}(B_{k,n})$ of Cohen-Macaulay modules over a quotient $B_{k,n}$ of a preprojective algebra of affine type $A$ provides an additive categorification of $\CC[\Gr_{k,n}]$ and they showed that there is a cluster character 
on ${\rm CM}(B_{k,n})$ which sends rigid indecomposable modules to cluster variables in $\CC[\Gr_{k,n}]$. They proved these results by showing that the quotient of this category by a single projective-injective object is Geiss-Leclerc-Schroer's category ${\rm Sub} Q_k$ \cite{GLS} which categorifies the coordinate ring of the big cell in the Grassmannian $\Gr(k,n)$. In their paper, the authors associated the root system $\rJ_{k,n}$ to $\mathrm{CM}(B_{k,n})$.  
They pointed out that rigid indecomposable modules in ${\rm CM}(B_{k,n})$ seem to correspond to (real or imaginary) roots of $\rJ_{k,n}$.
Thus studying the roots of $\rJ_{k,n}$ will help to study the rigid indecomposable modules in Grassmannian cluster categories $\mathrm{CM}(B_{k,n})$ \cite{JKS} and cluster variables in Grassmannian cluster algebras $\mathbb{C}[\Gr_{k,n}]$ \cite{Sco06}.

In this paper, we give a characterization of the real positive roots in the root system $\rJ_{k,n}$. A real positive root $\gamma \in \rJ_{k,n}$ is said to have degree $d$ if when $\gamma$ is written as a linear combination of simple roots, the coefficient of $\beta$ in $\gamma$ is $d$. Degree $0$ positive roots are just positive roots of the natural root subsystem of type $\rA_{n-1}$ given by the nodes $\alpha_1,\dots, \alpha_{n-1}$. 
They are all of the form $\alpha_i + \cdots + \alpha_{j-1}$ for some $1\leqslant i<j\leqslant n$.

In an arbitrary root system there is a procedure to check whether a positive element of the root lattice is a real root: for a positive real root there exists a sequence of simple reflections which at each step lowers the height (and eventually leads to a simple root), see~\cite[Proposition~5.1(e)]{Kac}. However, there is no systematic way to find this sequence other than by trial and error.

Our main result is that if one realizes the root lattice $\mathbb{Z}\Delta$ as a sublattice of $\mathbb{Z}^n$ (see \cref{sec:real roots in Ekn}), then in terms of the ambient lattice the above procedure can be done much faster and easier, as follows. 
\begin{restatable*}{theorem}{maintheoremrestate}
%\label{thm:main theororem}
\label{theorem:positive real roots of degree d}
$x = (x_1,\ldots,x_n)^\top \in\mathbb{Z}\Delta$ is a positive real root of degree $\geqslant 1$ if and only if
\begin{enumerate}
\item $0 \leqslant x_i \leqslant \deg(x)$ for all $i=1,\ldots,n$, \label{item:0<x<d}
\item $q(x) = 2$, \label{item:q(x)=2}
\item repeated application of $x\longmapsto s_\beta(\dec(x))$ preserves property (\ref{item:0<x<d}) until it changes the sign of all entries of $x$. \label{item:w(dec(x))}
\end{enumerate}
\end{restatable*}

Here $q(x)$ is a quadratic form (\ref{eq:q(x)}) on $\mathbb{Z}^n$, $s_{\beta}$ is the simple reflection associated with $\beta$,
\[ s_{\beta}\left((x_1,\ldots,x_n)^\top\right) = (x_1+r,\ldots,x_k+r,x_{k+1},\ldots,x_n)^\top, \]
$r = x_{k+1} + \ldots + x_n - 2\deg(x)$, $\deg(x) = \frac{1}{k}\sum_{i=1}^n x_i$, and 
$\dec(x)\in\mathbb{Z}\Delta$ is the element obtained from permuting the entries of $x_i$ to have them 
in decreasing order, i.e. if $\dec(x)=(x_1',\dots, x_n')$ then 
$x_1'\geqslant x_2'\geqslant \ldots \geqslant x_n'$.

The procedure in \cref{theorem:positive real roots of degree d} allows a very efficient enumeration of real roots. This enumeration reveals many regularities which are otherwise harder to see. Among other things, it provides another view on Manin's ``hyperbolic construction'' of $\rE_8$, which can be seen as the inclusion $\rE_8\subset \rJ_{4,9}$. This also highlights the connection between the affine roots of $\rE_9$ inside $\rJ_{4,10}$  and the exceptional curves on del Pezzo surfaces.

Jensen, King and Su conjectured \cite{JKS} that for every indecomposable module $M$ in ${\rm CM}(B_{k,n})$, there is a corresponding real or imaginary root $\varphi(M)$ (see Section \ref{sec:cluster algebras} for the definition of $\varphi(M)$) in the root system $\rJ_{k,n}$. It is conjectured in \cite[Conjecture~5.8]{BBGL} that whenever $M$ in ${\rm CM}(B_{k,n})$ is rigid indecomposable and $\varphi(M)$ is a real root in $\rJ_{k,n}$, then the profile $P_M$ (a profile is a certain array of integers, see Section \ref{sec:cluster algebras} for the definition) is a cyclic permutation of a canonical profile. The results about real roots in $\rJ_{k,n}$ in Theorem \ref{theorem:positive real roots of degree d} are thus expected to help with the characterization of rigid indecomposable modules in ${\rm CM}(B_{k,n})$ corresponding to real roots.

The paper is organized as follows. In \cref{sec:real roots in Ekn} we construct the root lattice and the action of the Weyl group on it, and give a characterization of real roots. In \cref{sec:symmetries and embeddings} we note various relations between the root systems $\rJ_{k,n}$ for distinct $k,n$. \cref{sec:enumeration of roots} is devoted to the enumeration of real roots and to some particular families of real roots.
In \ref{sec:systems of finite types} we discuss the finite types, 
i.e. the types  
$\rA_n$, $\rD_n$, $\rE_6$, $\rE_7$ and $\rE_8$ and give a simple description of the fundamental weights. \cref{sec:affine roots} provides a simpler description of isotropic roots in root systems of affine types $\rJ_{3,9}=\rE_8^{(1)}$ and $\rJ_{4,8}=\rE_7^{(1)}$. Also, in \cref{sec:almost real roots} we introduce the notion of ``almost real roots''. These are not 
roots but closely resemble the real roots. \cref{sec:hyperbolic construction} compares the description given in the present paper with Manin's ``hyperbolic construction'' of $\rE_8$. \cref{sec:cluster algebras} describes in greater details the connection to the cluster structures on the coordinate rings of Grassmannians mentioned above.

\subsection*{Acknowledgments} We would like to thank Alastair King for very helpful discussions. 
We also thank the anonymous referee for their work and for their helpful comments. 
K. B. was supported by a Royal Society Wolfson Fellowship 
RSWF/R1/180004 and by the EPSRC Programme Grant EP/W007509/1. 
She is currently on leave from the University of Graz.
She would like to thank the Isaac Newton Institute for Mathematical Sciences, Cambridge, for support and hospitality during the programme CAR 
where work on this paper was undertaken. This was supported by EPSRC grant no EP/R014604/1. 
 J.-R.L. was supported by the Austrian Science Fund (FWF): M 2633-N32 Meitner Program and P 34602 Einzelprojekte. A.S. was supported by Russian Science Foundation (RSF) (project No. 17-11-01261).

\section{Real roots in $\rJ_{k,n}$ root system} \label{sec:real roots in Ekn}
\subsection{Root lattice}\label{ssec:root-lattice}

Jensen, King, and Su gave a description of the root system $\rJ_{k,n}$ \cite[Section 2]{JKS}. This description of the root system arises naturally as the lattice that grades the Grassmannian cluster algebra ${\rm Gr}_{k,n}$. They observed that, for the right quadratic form $q(x)$ there seems to be a relationship between cluster variables and positive degree roots. We recall their results in the following.

Let $n$ and $k<n$ be two natural numbers and let $e_1,\ldots,e_n$ be the standard basis of $\mathbb{R}^n$. Let 
\[ \alpha_i = e_{i+1} - e_i\quad \text{for}\quad i=1,\ldots,n-1,\qquad \beta = e_1 + \ldots + e_k, \]
and consider the lattice
\[ \mathbb{Z}\Delta = \langle \beta,\alpha_1,\ldots,\alpha_{n-1}\rangle_{\mathbb{Z}} = \left\{ (x_1,\ldots,x_n)^\top \in \mathbb{Z}^n \ \middle|\ k\ \text{divides}\ x_1+\ldots+x_n\right\}, \]
called the \emph{root lattice} and equipped with the quadratic form
\begin{align}
q(x) = \sum_{i=1}^n x_i^2 + \frac{2-k}{k^2}\bigg(\sum_{i=1}^n x_i\bigg)^2. \label{eq:q(x)}
\end{align}
and an inner product given by its polarization $(x,y)=\frac{1}{2}\big(q(x+y)-q(x)-q(y)\big)$.

Any element $\alpha\in\mathbb{Z}\Delta$ can be written as $\alpha = m_\beta\beta + m_1\alpha_1+\ldots+m_{n-1}\alpha_{n-1}$, and the coefficient $m_\beta$ is called the \emph{degree} of $\alpha$, denoted by $\deg(\alpha)$. If $x = (x_1,\ldots,x_n)^\top\in\mathbb{Z}\Delta$, then
\[ \deg(x) = (x_1+\ldots+x_n)/k \qquad \text{and} \qquad q(x) = x_1^2+\ldots+x_n^2+(2-k)\deg(x)^2. \]

A direct calculation shows that the inner products are $(\alpha_i,\alpha_i)=2$ for $i=1,\dots, n-1$ 
and $(\beta,\beta)=2$ and that 
\[ (\alpha_i,\alpha_j) = \begin{cases}
%, & \text{if } i=j, \\
-1, & \text{if } |i-j|=1, \\
0, & \text{otherwise,}
\end{cases} \qquad
(\beta,\alpha_i) = \begin{cases}
-1, & \text{if } i=k, \\
0, & \text{otherwise},
\end{cases} %\qquad
%(\beta,\beta) = 2,
 \]
so that the Gram matrix $A$ of this inner product with respect to the basis 
$\beta,\alpha_1,\ldots,\alpha_{n-1}$ 
is the generalized Cartan matrix of the root system of type $\rJ_{k,n}$.

The matrix of the basis change from $\{\beta,\alpha_1,\dots, \alpha_{n-1}\}$ to $\{e_1,\dots, e_n\}$ is
\[
C = \begin{pmatrix}
1 & -1 \\1 & 1 & -1 \\
\vdots && 1 & -1 \\
1 &&& 1 & \ddots \\
0 &&&& \ddots & -1 \\
\vdots &&&&& 1 & -1 \\
0 &&&&&& 1
\end{pmatrix}.
\]
Therefore if 
\begin{gather*}
\gamma=m_\beta \beta+m_1\alpha_1+\ldots+m_{n-1}\alpha_{n-1}, \\
\shortintertext{then}
\gamma = x_1e_1+\ldots+x_ne_n, \\
\shortintertext{where}
(x_1,\ldots,x_n)^\top = C\cdot(m_\beta,m_1,\ldots,m_{n-1})^\top.
\end{gather*}

The inverse for $C$ is calculated as $C^{-1} = VDU$, where $V$ is the lower triangular matrix that differs from the identity matrix only in its first column, which equals
\[ (1,1-k,2-k,\ldots,-1,0,\ldots,0)^\top, \]
$D=\diag(\frac{1}{k},1\ldots,1)$, and $U$ is an upper triangular matrix having $1$ in all of its entries on and above the diagonal. Now
\[ Ux=(x_1+\ldots+x_n, x_2+\ldots+x_n,\ldots,x_n)^\top, \]
so the first entry equals $kd$, where $d$ is the degree of $x$. Thus the first entry of $DUx$ equals $d$, and the same holds for $VDUx$.

For the other entries of $VDUx$, note that
$x_i+\ldots+x_n$ can be rewritten as $kd-x_1-\ldots-x_{i-1}$. Thus for $i=2,\ldots,k$ the $i$-th entry of $VDUx$ equals
\[ (kd-x_1-\ldots-x_{i-1}) + (i-k-1)d = (i-1)d-x_1-\ldots-x_{i-1}. \]
In particular, for $i=k$ it is
\[ (k-1)d-x_1-\ldots-x_{k-1} = x_k+\ldots+x_n-d. \]
For $i>k$ the $i$-th entry is the same as the $i$-th entry of $Ux$, that is, $x_i+\ldots+x_n$.

Many standard notions can be directly expressed in terms of the $e_i$ basis. For example, if $\gamma=d\beta+\sum m_i\alpha_i$, the scalar product $(\beta,\gamma)=2d-m_k$ can be computed as $2d-(x_{k+1}+\ldots+x_n)$.

\begin{example} \label{example:dn}
We demonstrate how the correspondence between the two bases described above works in the second simplest case, that is in the case $k=2$ and $\rJ_{k,n}=D_n$.
In $\rD_n$ the roots of degree $1$ are of one of the following three forms 
(written in the $e_i$'s on the left and in terms of the simple roots on the right): 
\begin{align*}
 & e_1+e_j \quad \leftrightsquigarrow \quad \eroot{0&1&1&\ldots&1&0&\ldots&0\\&1&&&\hspace*{-1.5ex}\mathrlap{\lcurvearrownw j-1}} \quad\quad \text{with}\ 2\leqslant j \leqslant n-1, \\
& e_2+e_j \quad \leftrightsquigarrow \quad \eroot{1&1&1&\ldots&1&0&\ldots&0\\&1&&&\hspace*{-1.5ex}\mathrlap{\lcurvearrownw j-1}} \quad\quad \text{with}\ 3\leqslant j \leqslant n-1, \\
& e_i+e_j \quad \leftrightsquigarrow \quad \eroot{1&2&2&\ldots&2&1&\ldots&1&0&\ldots&0\\&1&&&\mathllap{i-1\rcurvearrowne}\hspace*{-1.5ex}&&&\hspace*{-1.5ex}\mathrlap{\lcurvearrownw j-1}}\quad \text{with}\ 3\leqslant i<j\leqslant n-1.
\end{align*}
So in this case, the description of the positive roots in terms of the $e_i$'s coincides (up to the ordering of the simple roots) with the standard realization~\cite[Ch.~VI, \S4, no.~8]{Bou1}: the positive roots of $\rD_n$ are
%(\revise{should we say: the set of positive roots in $\rD_n$ is (because we didn't introduce the notation $\rD_n^+$)?} \revise{In the above, there is no $e_i - e_j$, but in the following there are $e_i-e_j$.})
%(\kb{I agree - write ``the positive roots of $\rD_n$ are...''.})
\begin{gather*}
e_i\pm e_j \quad\text{with}\quad 1\leqslant i<j\leqslant n. \\
\shortintertext{and the simple roots are}
\alpha_i = e_{i+1}-e_i,\quad 1\leqslant i\leqslant n-1, \qquad \beta = e_1+e_2.
\end{gather*}
Note that in~\cite[Ch.~VI, \S4, no.~8]{Bou1} the numbering of the simple roots is reversed, and $\beta$ is attached to $\alpha_{n-2}$.
\end{example}

\subsection{Weyl group action}\label{ssec:Weyl_group}
For $i=1,\ldots,n-1$ denote by $s_i$ the involution on $\mathbb{Z}^n$ induced by the transposition $(i, i+1)$ on the basis $e_1,\ldots,e_n$, and its restriction on $\mathbb{Z}\Delta$. Denote also by $s_\beta$ the linear map
\[ x = (x_1,\ldots,x_n)^\top \longmapsto (x_1+r,\ldots,x_k+r,x_{k+1},\ldots,x_n)^\top, \]
where $r = x_{k+1} + \ldots + x_n - 2\deg(x)$. The map $s_\beta$ acts on $\mathbb{Z}\Delta$, indeed, if $x\in\mathbb{Z}\Delta$, then
\[ (x_1+r)+\ldots+(x_k+r) + x_{k+1}+\ldots+x_n = \sum x_i + kr \]
is also divisible by $k$, and $\deg(s_\beta x) = \deg(x) + r$.
\begin{lemma}
The above formulas define an action of the Weyl group $W(\rJ_{k,n})$ on $\mathbb{Z}\Delta$.
\end{lemma}
\begin{proof}
To show that the action of $s_\beta,s_1,\ldots,s_n$ define an action of the Weyl group, it is enough to show that they satisfy the defining relations of $W(\rJ_{k,n})$ in its standard presentation as a Coxeter group, that is
\begin{align*}
& s_\beta^2 = s_1^2 = \ldots = s_{n-1}^2 = \operatorname{id}, \\
& s_i s_j s_i = s_j s_i s_j \quad \text{for} \quad |i-j|=1, \quad i,j\leqslant n-1, \\
& s_i s_j=s_j s_i \quad \text{for} \quad |i-j|>1, \quad i,j\leqslant n-1, \\
& s_\beta s_k s_\beta = s_k s_\beta s_k, \\
& s_\beta s_i = s_i s_\beta \quad \text{for} \quad i\neq k.
\end{align*}
Note that the relations not involving $s_\beta$ are satisfied because $s_1,\ldots,s_{n-1}$ are defined as the fundamental transpositions, which are known to be the standard Coxeter generators of the permutation group $S_n$ \cite[Section~2.8.1]{Wilson}.

Now set $x=(x_1,\ldots,x_n)\in\mathbb{Z}\Delta$. Then
\[ x \xmapsto{\ \ \mathclap{s_\beta} \ \ } (x_1+r,\ldots,x_k+r,x_{k+1},\ldots,x_n)^\top \xmapsto{\ \ \mathclap{s_\beta} \ \ } (x_1+r+r',\ldots,x_k+r+r',x_{k+1},\ldots,x_n)^\top, \]
where $r = x_{k+1}+\ldots+x_n-2\deg(x)$ and $r' = x_{k+1}+\ldots+x_n-2\deg(s_\beta x)$. But $\deg(s_\beta x) = \deg(x)+r$, hence $r'=r-2r$, so $s_\beta^2(x)=x$.

The commutation relation for $s_\beta$ and $s_i$, $i\neq k$ are obvious from the definition.

To prove the braiding relation for $s_\beta$ and $s_k$ consider
\begin{align*}
x & \xmapsto{\ \ \mathclap{s_k} \ \ } (x_1,\ldots,x_{k-1},x_{k+1},x_k,x_{k+2},\ldots,x_n)^\top \\
& \xmapsto{\ \ \mathclap{s_\beta} \ \ } (x_1+r,\ldots,x_{k-1}+r,x_{k+1}+r,x_k,x_{k+2},\ldots,x_n)^\top \\
& \xmapsto{\ \ \mathclap{s_k} \ \ } (x_1+r,\ldots,x_{k-1}+r,x_k,x_{k+1}+r,x_{k+2},\ldots,x_n)^\top,
\end{align*}
where $r = x_k+x_{k+2}+\ldots+x_n-\deg(x)$, and
\begin{align*}
x & \xmapsto{\ \ \mathclap{s_\beta} \ \ } (x_1+t,\ldots,x_k+t,x_{k+1},\ldots,x_n)^\top \\
& \xmapsto{\ \ \mathclap{s_k} \ \ } (x_1+t,\ldots,x_{k-1}+t,x_{k+1},x_k+t,x_{k+2},\ldots,x_n)^\top \\
& \xmapsto{\ \ \mathclap{s_\beta} \ \ } (x_1+t+t',\ldots,x_{k-1}+t+t',x_{k+1}+t',x_k+t,x_{k+2},\ldots,x_n)^\top,
\end{align*}
where $t = x_{k+1}+\ldots+x_n-2\deg(x)$ and $t' = x_k+t+x_{k+2}+\ldots x_n-2\deg(s_k s_\beta x)$. But $\deg(s_k s_\beta x) = \deg(x) + t$, so $t' = x_k - x_{k+1}$. Thus $r=t+t'$, and also $x_{k+1}+t'=x_k$ and $x_k+t=x_{k+1}+r$, which means that $s_k s_\beta s_k(x) = s_\beta s_k s_\beta(x)$.
\end{proof}
\begin{lemma} \label{lemma:q is invariant}
$W(\rJ_{k,n})$ acts on $\mathbb{Z}\Delta$ by isometries, that is, $q(wx)=q(x)$ for any $w\in W(\rJ_{k,n})$.
\end{lemma}
\begin{proof}
The quadratic form $q$ is invariant under the action of $s_i$, because $q$ is defined in terms of symmetric polynomials. Concerning the action of $s_\beta$, denote $r = x_{k+1} + \ldots +x_n - 2d$, where $d = \deg(x)$, and consider
\begin{align*}
q(s_\beta x) & = \sum_{i=1}^k (x_i+r)^2 + \sum_{\mathclap{i=k+1}}^n x_i^2 + \frac{2-k}{k^2}\bigg( kr + \sum_{i=1}^n x_i \bigg)^2 \\
& = \sum_{i=1}^n x_i^2 + 2r\bigg(\sum_{i=1}^k x_i\bigg) + kr^2 + (2-k)(d^2+2dr+r^2) \\
& = q(x) + 2r\bigg( \sum_{i=1}^k x_i + 2d - kd + r \bigg) \\
& = q(x) + \bigg( r + 2d - \sum_{\mathclap{i=k+1}}^n x_i \bigg) = q(x). \qedhere
\end{align*}
\end{proof}
\begin{remark}
This action is faithful and coincides with the standard action of $W(\rJ_{k,n})$ on the root lattice.
\end{remark}
\begin{proof}
Straightforward check for the action of the generators on the simple roots.
\end{proof}

\subsection{Real roots and degree change}
The set of real roots $\Delta_{\mathrm{re}}$ of the root system $\Delta$ is defined as the union of the Weyl group orbits of its simple roots. Since $\rJ_{k,n}$ is simply-laced, all simple roots lie in the same orbit, and so $\Delta_{\mathrm{re}} = W(\rJ_{k,n})\beta$. Recall that in the basis $e_1,\ldots,e_n$ one has $\beta=(1,\ldots,1,0,\ldots,0)^\top$.

The set of positive real roots is denoted by $\Delta_{\mathrm{re}}^+$.
\begin{remark} \label{remark:degree 0 roots}
Real roots of degree $0$ form a subsystem of type $\rA_{n-1}$ and are of the form $e_i-e_j$, $i\neq j$. The root $e_i-e_j$ is positive if $i>j$.
\end{remark}
\begin{lemma} \label{lemma:0<x<d is preserved}
If $x = (x_1,\ldots,x_n)^\top \in\Delta_{\mathrm{re}}^+$ and $\deg(x) = d \geqslant 1$, then $0\leqslant x_i\leqslant d$ for all $i=1,\ldots,n$.
\end{lemma}
\begin{proof}
Induction by $\deg(x)$. 

Consider first the case $\deg(x)=1$. This means that $x_1+\ldots+x_n=k$. On the other hand, for a real root $x$ one has $q(x)=2$. But
\[ q(x) = \sum x_i^2 + \frac{2-k}{k^2} k^2 = \sum x_i^2+2-k, \]
hence $x_1^2+\ldots+x_n^2=k$. It follows that all $x_i$ are either $0$ or $1$ (otherwise $\sum x_i^2 > \sum x_i$).

Now if $x\in\Delta_{\mathrm{re}}$ has $\deg(x)>1$, then there is $y\in\Delta_{\mathrm{re}}^+$ such that $\deg(y)<\deg(x)$ and $x = w(y)$ for some $w\in W(A)$. The assumption $\deg(y)<\deg(x)$ holds because $w$ can be chosen to be a sequence of reflections which only increase the height, see~\cite[Proposition~1]{Moody}.

Since $w = \sigma_1 s_\beta \sigma_2 \dots s_\beta \sigma_m$ for some $\sigma_1,\ldots,\sigma_m \in W(\rA_{n-1}) \cong S_n$, one can replace $y$ by $=\sigma_2 s_\beta\dots\sigma_m(y)$ and $x$ by $\sigma_1^{-1}(x)$, so that $x=s_{\beta}(y)$.

Now $y = s_\beta(x) = (x_1+r,\ldots,x_k+r,x_{k+1},\ldots,x_n)^\top$ with $r<0$, and $\deg(y) = \deg(s_\beta x) = \deg(x)+r$. By the induction hypothesis, $0\leqslant x_i+r\leqslant\deg(y)$ for $i=1,\ldots,k$, and $0\leqslant x_i\leqslant \deg(y)$ for $i>k$, so $0\leqslant x_i\leqslant\deg(x)$ for all $i$.
%(\kb{not sure I understand: I can see that $w_\beta(x)=y=(x_1+r,\dots)^\top$ but why does this imply 
%that the $x_i$ are $\le d$?}) 
%so the desired inequalities obviously hold.
\end{proof}
\begin{remark}
For $x\in\mathbb{Z}\Delta$, which is not a real root, the inequalities $0\leqslant x_i\leqslant \deg(x)$ are only guaranteed to be preserved by $s_\beta$ in case this action increases the degree. That is, if $0\leqslant x_i\leqslant \deg(x)$, $\deg(x)>0$ and $0<\deg(s_\beta(x))<\deg(x)$, then $s_\beta(x)$ can have negative entries or entries greater than its degree. See \cref{sec:almost real roots} for examples.
\end{remark}
We will now establish some conditions on $x\in\mathbb{Z}\Delta$ which guarantee that $s_\beta$ lowers the degree of $x$.
\begin{lemma} \label{lemma:w reduces degree}
If $x=(x_1,\ldots,x_n)^\top\in\mathbb{R}^n$ and $d\in\mathbb{R}$ satisfy
\begin{align*}
& d\geqslant x_1\geqslant\ldots\geqslant x_n\geqslant0,\\
& x_1+\ldots+x_n=kd,\\
& q(x)=\sum x_i^2+(2-k)d^2>0,
\end{align*}
then $x_{k+1}+\ldots+x_n < 2d$.
\end{lemma}
\begin{proof}
The statement of the lemma can be reformulated as follows. Consider the polyhedron $P$ given by the inequalities
\[ P = \left\{ x\in\mathbb{R}^n\ \middle|\ \begin{aligned}
& d\geqslant x_1\geqslant\ldots\geqslant x_n\geqslant0, \\
& x_{k+1}+\ldots+x_n\geqslant 2d, \\
& x_1+\ldots+x_n = kd
\end{aligned} \right\}. \]
One has to show that $P$ lies in a closed ball of radius $d\sqrt{k-2}$ centered at the origin.

Note that by scaling everything down $d$ times it is sufficient to prove this statement for $d=1$. Note also that the maximal distance from $0$ over all points of $P$ is attained at one of its vertices.

Denote by $A$ the following $(n+2){\times}n$ matrix, by $b$ the following column vector in $\mathbb{R}^{n+2}$ and by $c$ the following row vector in $\mathbb{R}^n$:
\[ A = \begin{pmatrix}
1 \\
-1 & 1 \\
& -1 & 1 \\
&& \ddots & \ddots \\
&&& -1 & 1\\
&&&& -1 \\
0\ \mathrlap{\ldots} & 0 & -1 & \ldots & -1
\end{pmatrix},\quad b = \begin{pmatrix}
1 \\ 0 \\\\ \vdots \\\\ 0 \\ -2
\end{pmatrix},\quad c = (1,\ldots,1), \]
so that $P = \{ x\in\mathbb{R}^n \mid Ax\leqslant b,\ c\cdot x = k \}$.

Denote also by $A(i,j,l)$ the square matrix of the form $\begin{pmatrix}\rA_{I,*}\\c\end{pmatrix}$, 
where $\rA_{I,*}$ is the submatrix of $A$ consisting of rows with indices in $I=\{1,\ldots,n\}\setminus\{i,j,l\}$.  
Finally, denote by $b(i,j,l)=\begin{pmatrix}b_I\\k\end{pmatrix}$.

Then the vertices of $P$ are those points of $P$ which satisfy the equation
\[ A(i,j,l)x=b(i,j,l), \]
where $1\leqslant i < j < l \leqslant n+2$ are such that $A(i,j,l)$ is non-singular.

First note that $i\leqslant k$. Indeed, if $i>k$, then $A(i,j,l)$ contains the first $i$ rows of $A$. The equation coming from $A_{1,*}$ implies $x_1=1$, while the next $i-1$ equations mean that $x_1=x_2=\ldots=x_i$, so $x_1+\ldots+x_n \geqslant x_1+\ldots+x_i = i > k$.

Note also that $i,j,l$ cannot all be simultaneously $\leqslant k+1$, because there is a linear dependence between the rows from $k+2$ to $n+2$, namely,
\[ \rA_{k+2,*}+2\rA_{k+3,*}+3\rA_{k+4,*}+\ldots+(n-k)\rA_{n+1,*} = \rA_{n+2,*}. \]

Assume first that $l=n+2$. Then $j\geqslant k+1$, because otherwise $x_{k+1}=\ldots=x_n=0$ and thus $x_{k+1}+\ldots+x_n<2$. The solution is of the form
\[ x = (\underbrace{1,\ldots,1}_{i-1},\underbrace{y,\ldots,y}_{j-i},0,\ldots,0), \quad y = \frac{k+1-i}{j-i}. \]
Since $j\geqslant k+1$, the inequalities of the form $x_s\geqslant x_{s+1}$ are also satisfied. Now $x_{k+1}+\ldots+x_n=y\cdot(j-k-1)$.

Let $s=k+1-i$, so that $y=\frac{s}{j-i}$. If $x$ is a vertex, then 
\[ x_{k+1}+\ldots+x_n=y\cdot(j-k-1) \geqslant 2, \]
or, equivalently,
\[ \frac{s}{j-i}(j-i-s) \geqslant 2. \]
The squared distance from the origin to $x$ equals
\begin{multline*}
x_1^2+\ldots+x_n^2 = (i-1)+(j-i)y^2 = (i-1)+\frac{(k+1-i)^2}{j-i} = \\ 
 = k-s+\frac{s^2}{j-i} = k+\frac{s}{j-i}(s+i-j) \leqslant k-2.
\end{multline*}

Now assume that $l<n+2$. Then the solution is of the form
\[ x = (\underbrace{1,\ldots,1}_{i-1}, \underbrace{y_1,\ldots,y_1}_{j-i}, \underbrace{y_2,\ldots,y_2}_{l-j}, 0,\ldots,0). \]
The values of $y_1$ and $y_2$ are subject to the following two equations. The first one is
\[ x_1+\ldots+x_n=k, \quad\text{i.e.}\quad (i-1)+y_1\cdot(j-i)+y_2\cdot(l-j) = k. \]
The form that the equation $x_{k+1}+\ldots+x_n=2$ takes depends on $j$.

If $j\leqslant k+1$, the second equation is $y_2\cdot(l-k-1) = 2$. In this case
\[ y_2 = \frac{2}{l-k-1},\quad y_1 = \frac{(k+1-i)(l-k-1)+2(j-l)}{(l-k-1)(j-i)}. \]
Write 
\begin{gather*}
s = k+1-i, \quad t = k+1-j, \quad r=l-k-1, \\
\shortintertext{so that}
y_1 = \frac{sr-2(t+r)}{r(s-t)}, \quad y_2 = \frac{2}{r}.
\end{gather*}
The inequality $y_2\leqslant y_1$ can be reformulated as $2(s+r)\leqslant sr$, while the inequality $y_1\leqslant 1$ means $rt \leqslant 2(t+r)$. Now
\[ x_1^2+\ldots+x_n^2 = k-s + (s-t)\frac{(sr-2(t+r))^2}{r^2(s-t)^2} + (t+r)\frac{4}{r^2}. \]
This sum being not greater than $k-2$ can be expressed as
\[ 2-s + \frac{(sr-2(t+r))^2}{r^2(s-t)} + \frac{4(t+r)}{r^2} \leqslant 0, \]
or, multiplying by $r^2(s-t)$,
\[ (2-s)(s-t)r^2 + (sr-2(t+r))^2 + 4(t+r)(s-t) \leqslant 0. \]
%2-s+\frac{s^2r^2-4sr(t+r)+4(t+r)^2+4(t+r)(s-t))}{r^2(s-t)} = 2-s+\frac{s^2}{s-t}+\frac{4(r+t)(s+r-sr)}{r^2(s-t)}
But the left-hand side can be rewritten as
\[ (sr-2(s+r))\cdot(rt-2(t+r)), \]
which is non-positive.

If $j>k+1$, the second equation becomes $y_1\cdot(j-k-1)+y_2\cdot(l-j)=2$. This system of equations is equivalent to
\begin{gather*}
\begin{pmatrix} j-i & l-j \\ i-k-1 & 0 \end{pmatrix}
\begin{pmatrix} y_1 \\ y_2 \end{pmatrix}
=
\begin{pmatrix} k+1-i \\ i-k+1 \end{pmatrix}, \\
\shortintertext{thus}
y_1 = \frac{i-k+1}{i-k-1},\quad y_2 = \frac{(k+1-i)(i-k-1)+(i-j)(i-k+1)}{(i-k-1)(l-j)}.
\end{gather*}
Again, write 
\begin{gather*}
s = k+1-i, \quad t = j-k-1, \quad r=l-k-1, \\
\shortintertext{so that}
y_1 = \frac{s-2}{s}, \quad y_2 = \frac{s^2+(s+t)(2-s)}{s(r-t)} = \frac{2(s+t)-st}{s(r-t)}.
\end{gather*}
The inequalities $0\leqslant y_2\leqslant y_1$ are equivalent to $st\leqslant 2(s+t)$ and $sr\geqslant 2(s+r)$. Now
\[ x_1^2+\ldots+x_n^2 = k-s + (s+t)\frac{(s-2)^2}{s^2} + (r-t)\frac{(2(s+t)-st)^2}{s^2(r-t)^2}, \]
and this sum being not greater than $k-2$ is equivalent to
\[ (2-s)(r-t)s^2 + (s+t)(s-2)^2(r-t) + (2(s+t)-st)^2 \leqslant 0. \]
The left-hand side can be rewritten as
\[ (st-2(s+t))\cdot(sr-2(s+r)), \]
which is non-positive.
\end{proof}

Denote by $\dec(x)\in\mathbb{Z}\Delta$ the permutation of entries of $x=(x_1,\ldots,x_n)^\top$ such that $x_1\geqslant x_2\geqslant\ldots\geqslant x_n$.
\begin{corollary} \label{cor:degree reduced for real roots}
If $x \in\Delta_{\mathrm{re}}^+$ is such that $x=\dec(x)$, then $\deg(s_\beta x) < \deg(x)$.
\end{corollary}
\begin{proof}
The entries of the real root $x = (x_1,\ldots,x_n)^\top$ satisfy the assumptions of the 
Lemma~\ref{lemma:w reduces degree}, 
so $\deg(s_\beta x) = \deg(x) + x_{k+1} + \ldots x_n - 2\deg(x) < \deg(x)$.
\end{proof}

\maintheoremrestate
%\begin{theorem} \label{theorem:positive real roots of degree d}
%The vector 
%$x = (x_1,\ldots,x_n)^\top \in\mathbb{Z}\Delta$ is a positive real root of degree $\geqslant 1$ if and only if
%\begin{enumerate}
%\item $0 \leqslant x_i \leqslant \deg(x)$ for all $i=1,\ldots,n$, \label{item:0<x<d}
%\item $q(x) = 2$, \label{item:q(x)=2}
%\item repeated application of $x\longmapsto w_\beta(\dec(x))$ preserves property (\ref{item:0<x<d}) until it changes the sign of all entries of $x$. \label{item:w(dec(x))}
%\end{enumerate}
%\end{theorem}
\begin{proof}
For any real root $x$ with $\deg(x)\geqslant1$, all three properties are satisfied, because they are 
satisfied for $\beta=(1,\ldots,1,0,\ldots,0)^\top$ and are preserved by the operation $x\longmapsto s_\beta(\dec(x))$ by \cref{lemma:0<x<d is preserved,lemma:q is invariant,cor:degree reduced for real roots}.

Now assume that $x\in\mathbb{Z}\Delta$ satisfies these three properties, and consider the sequence
\[ 
\big(x^{(i)}\big)_{i\in\mathbb{Z}_{\geqslant0}} \quad \text{with} \quad x^{(0)} = x, \quad x^{(i+1)} = s_\beta\big(\dec(x^{(i)})\big). 
\]
By \cref{lemma:w reduces degree} $\deg(x^{(i+1)}) < \deg(x^{(i)})$. Denote by $m$ the smallest index such that $x^{(m)}$ has non-negative entries but $x^{(m+1)}$ only has non-positive entries. 
Then $\dec(x^{(m)})$ is of the form $(y_1,\ldots,y_k,0,\ldots,0)$ for some non-negative $y_i$, because $s_\beta$ only affects the first $k$ entries and must change the signs of all entries by property~(\ref{item:w(dec(x))}).

Denote $d=\deg(x^{(m)})$, so that by property~(\ref{item:w(dec(x))}) $y_i\leqslant d$. Now since $(y_1+\ldots+y_k)/k=d$, it follows that $y_i=d$ and $\dec(x^{(m)}) = (d,\ldots,d,0,\ldots,0)^\top$. But then
\[ q\big(\dec(x^{(m)})\big) = kd^2+(2-k)d^2 = 2d^2, \]
so by property~(\ref{item:q(x)=2}) $d=1$ and $\dec(x^{(m)}) = \beta$ (and $x^{(m+1)}=-\beta$). This implies that $x \in W\beta$ and hence it is a real root.
\end{proof}
\begin{remark}
It follows from the proof of the above theorem that property~(\ref{item:w(dec(x))}) can be replaced 
by the following: repeated application of $x\longmapsto s_\beta(\dec(x))$ leads to $-\beta$, and it does so in at most $\deg(x)$ steps (in particular, for a real root $x$ of degree $1$ one has $\dec(x)=\beta$ and $s_\beta(\dec(x))=-\beta$).
\end{remark}
\begin{remark}
The process described above also works for the elements of the root lattice close to the real roots. In particular, it allows to distinguish non-roots admitting a sequence of height-lowering simple reflections, see~\cref{sec:almost real roots}. The latter are related to the indecomposable modules appearing in the categorification of Grassmannian cluster algebras, see~\cref{sec:cluster algebras}.
\end{remark}

\section{Symmetries and embeddings} \label{sec:symmetries and embeddings}
There is a natural correspondence between $\rJ_{k,n}$ and $\rJ_{n-k,n}$ as their graphs are isomorphic, 
i.e. they have the same root system, with a different ordering of the simple roots.
\begin{remark}\label{remark:dual-root-system}
If $x = (x_1,\ldots,x_n)^\top$ is an element of the root lattice $\rJ_{k,n}$, 
then the corresponding element of the root lattice $\rJ_{n-k,n}$ is 
$x' = (d-x_n,\ldots,d-x_1)^\top$, where $d=\deg(x)$.
\end{remark}
\begin{proof}
This correspondence is linear, maps simple roots to simple roots in symmetric positions and preserves the quadratic form:
\begin{align*}
q_{n-k,k}(x') & = \sum(d-x_i)^2+\big(2-(n-k)\big)d^2 \\
& = nd^2-2d\cdot\sum x_i + \sum x_i^2 + (2-n+k)d^2 \\
& = \sum x_i^2 - 2kd^2 + (2+k)d^2 = q_{k,n}(x). \qedhere
\end{align*}
\end{proof}

The root system $\rJ_{k,n}$ can be considered as a subsystem of both $\rJ_{k,n+1}$ and $\rJ_{k+1,n+1}$ in the natural way, meaning that the branch node of the tree is mapped to the branch point of the larger graph. In terms of the $\alpha_i, \beta$, we can consider any positive root for a larger system containing it. The next remark explains how the subsystem arises in terms of the $x_i's$.

\begin{remark}\label{remark:root-system-extension}
If $x = (x_1,\ldots,x_n)^\top$ is an element of the root lattice $\rJ_{k,n}$ of degree $d$, then 
the corresponding elements of the root latices $\rJ_{k,n+1}$ and $\rJ_{k+1,n+1}$ are
\[ 
(x_1,\ldots,x_n,0)^\top\quad \text{and} \quad (d,x_1,\ldots,x_n)^\top, 
\]
respectively. Note that both have degree $d$ in their respective root lattice.
\end{remark}
\begin{proof}
Both correspondences are linear, map simple roots to the respective simple roots and preserve the quadratic form:
\[ d^2+\sum x_i^2 + \big(2-(k+1)\big)d^2 = \sum x_i^2 + (2-k)d^2. \qedhere \]
\end{proof}
In particular, iterating the above, this provides a description of
the infinite rank root system $\rJ_{\infty,\infty}$ as a set of $\mathbb{Z}$-indexed sequences $(x_i)_{i\in\mathbb{Z}}$ of integers such that there 
exists $M\in\mathbb{N}$ such that
\begin{enumerate}
\item $(x_{-M},\ldots,x_M)^\top$ is an element of $\rJ_{M,2M+1}$ of degree $d$,
\item $x_{M+1} = x_{M+2} = \ldots = 0$,
\item $x_{-M-1} = x_{-M-2} = \ldots = d$.
\end{enumerate}
Note that the particular choice of such $M$ does not change the degree $d$. Indeed, if $x_{-M}+\ldots+x_M=Md$, then $x_{-M-m}+\ldots+x_{M+m} = (M+m)d$.

This naturally extend to the description of the inifinite rank root lattice $\mathbb{Z}\rJ_{\infty,\infty}$. It is equipped with the inner product $q$ defined as the value of $q$ on $\mathbb{Z}\rJ_{M,2M+1}$ for a suitable $M$.

We also define $\dec(x)$ as follows: for $(x_i)_{i\in\mathbb{Z}}\in\mathbb{Z}\rJ_{\infty,\infty}$ with $x_i\geqslant 0$ and of degree $d$ denote $x' = \dec\left( (x_{-M},\ldots,x_{M})^\top \right)\in\mathbb{Z}\rJ_{M,2M+1}$ and set $\dec(x)$ to be the image of $x'$ in $\mathbb{Z}\rJ_{\infty,\infty}$. Note that if $x$ is a real root of $\rJ_{\infty,\infty}$, then so is $\dec(x)$.

For every $x\in\mathbb{Z}\rJ_{\infty,\infty}$ there exist the smallest $k,n$ such that $x$ comes from an element of $\mathbb{Z}\rJ_{k,n}$ by means of \cref{remark:root-system-extension}. Namely, $k$ is the smallest natural number such that $x_{-k-1}=x_{-k-2}=\ldots=\deg(x)$, while $n$ is the smallest such that $x_{n-k}=x_{n-k+1}=\ldots=0$.

\section{Enumeration of roots} \label{sec:enumeration of roots}
\subsection{Real roots}
In order to enumerate roots, we can use the action of the $\rA_{n-1}$ type subsystem in $\rJ_{k,n}$ as 
explained in the following remark.

\begin{remark}\label{remark:W-orbits}
The Weyl group $W(\rA_{n-1})\cong S_n$ of the root subsystem 
$\langle \alpha_1,\ldots,\alpha_{n-1}\rangle$ of type $\rA_{n-1}$ acts on the root lattice by permutations on the entries of $x$ while keeping the degree. 
Thus the enumeration of roots reduces to the enumeration of the orbits of this action on the roots of 
each degree. This will be our strategy in this section.
In each such orbit we choose one representative which is ordered. This representative has 
the smallest height over its orbit, and since the support of a root is connected, this root belongs to a 
particular natural subsystem of the smallest rank.
\end{remark}

All orbits of real roots of degrees up to $5$ (assuming $k$ and $n$ are large enough) are listed in \cref{table:real-roots}, in the coordinates $(x_i)_i$.
\begin{table}[h]
\begin{tabular}{l@{\qquad}l@{\qquad}l} \toprule
degree $1$ & \multicolumn{2}{l}{\hspace{-1ex}$\ubr{1\ldots1}_{k}$ (\ref{eq:root-type-E39-1}\textsuperscript{$+$},\ref{eq:root-type-E39-2}\textsuperscript{$-$},\ref{eq:root-type-E69-1}\textsuperscript{$+$},\ref{eq:root-type-E69-2}\textsuperscript{$-$},\ref{eq:root-type-E48-1}\textsuperscript{$\pm$},\ref{eq:root-type-three},\ref{eq:root-type-sym})}\\ \midrule
degree $2$ & \multicolumn{2}{l}{\hspace{-1ex}$\ubr{2\ldots2}_{k-3}\ubr{1\ldots1}_{6}$ \quad (\ref{eq:root-type-E39-1}\textsuperscript{$-$},\ref{eq:root-type-E39-2}\textsuperscript{$+$},\ref{eq:root-type-E69-1}\textsuperscript{$-$},\ref{eq:root-type-E69-2}\textsuperscript{$+$},\ref{eq:root-type-E48-0}\textsuperscript{$\pm$},\ref{eq:root-type-E48-2}\textsuperscript{$\pm$},\ref{eq:root-type-three},\ref{eq:root-type-sym})} \\ \midrule
degree $3$ &
\multicolumn{2}{l}{\hspace{-1ex}$\ubr{3\ldots3}_{k-3}\ubr{2}_{1}\ubr{1\ldots1}_{7}$ (\ref{eq:root-type-E39-0}\textsuperscript{$\pm$},\ref{eq:root-type-E39-3}\textsuperscript{$\pm$},\ref{eq:root-type-three})} \\ &
\multicolumn{2}{l}{\hspace{-1ex}$\ubr{3\ldots3}_{k-4}\ubr{2\ldots2}_{4}\ubr{1\ldots1}_{4}$ (\ref{eq:root-type-E48-1}\textsuperscript{$\pm$},\ref{eq:root-type-sym})} \\ &
\multicolumn{2}{l}{\hspace{-1ex}$\ubr{3\ldots3}_{k-5}\ubr{2\ldots2}_{7}\ubr{1}_{1}$ (\ref{eq:root-type-E69-0}\textsuperscript{$\pm$},\ref{eq:root-type-E69-3}\textsuperscript{$\pm$},\ref{eq:root-type-three}')} \\ \midrule
degree $4$ &
$\ubr{4\ldots4}_{k-3}\ubr{3}_{1}\ubr{1\ldots1}_{9}$ (\ref{eq:root-type-three}) &
$\ubr{4\ldots4}_{k-3}\ubr{222}_{3}\ubr{1\ldots1}_{6}$ (\ref{eq:root-type-E39-1}\textsuperscript{$+$},\ref{eq:root-type-E39-2}\textsuperscript{$-$}) \\ &
$\ubr{4\ldots4}_{k-4}\ubr{33}_{2}\ubr{222}_{3}\ubr{1\ldots1}_{4}$ &
$\ubr{4\ldots4}_{k-4}\ubr{3}_{1}\ubr{2\ldots2}_{6}\ubr{1}_{1}$ (\ref{eq:root-type-E48-0}\textsuperscript{$\pm$},\ref{eq:root-type-E48-2}\textsuperscript{$\pm$}) \\ &
$\ubr{4\ldots4}_{k-5}\ubr{3\ldots3}_{5}\ubr{1\ldots1}_{5}$ (\ref{eq:root-type-sym}) &
$\ubr{4\ldots4}_{k-5}\ubr{3\ldots3}_{4}\ubr{222}_{3}\ubr{11}_{2}$ \\ &
$\ubr{4\ldots4}_{k-6}\ubr{3\ldots3}_{6}\ubr{222}_{3}$ (\ref{eq:root-type-E69-1}\textsuperscript{$+$},\ref{eq:root-type-E69-2}\textsuperscript{$-$}) &
$\ubr{4\ldots4}_{k-7}\ubr{3\ldots3}_{9}\ubr{1}_{1}$ (\ref{eq:root-type-three}') \\ \midrule
degree $5$ &
$\ubr{5\ldots5}_{k-3}\ubr{4}_{1}\ubr{1\ldots1}_{11}$ (\ref{eq:root-type-three}) &
$\ubr{5\ldots5}_{k-3}\ubr{3}_{1}\ubr{222}_{3}\ubr{1\ldots1}_{6}$ \\ &
$\ubr{5\ldots5}_{k-3}\ubr{2\ldots2}_{6}\ubr{111}_{3}$ (\ref{eq:root-type-E39-1}\textsuperscript{$-$},\ref{eq:root-type-E39-2}\textsuperscript{$+$}) &
$\ubr{5\ldots5}_{k-4}\ubr{44}_{2}\ubr{2\ldots2}_{4}\ubr{1\ldots1}_{4}$ \\ &
$\ubr{5\ldots5}_{k-4}\ubr{4}_{1}\ubr{333}_{3}\ubr{2}_{1}\ubr{1\ldots1}_{5}$ &
$\ubr{5\ldots5}_{k-4}\ubr{4}_{1}\ubr{33}_{2}\ubr{2\ldots2}_{4}\ubr{11}_{2}$ \\ &
$\ubr{5\ldots5}_{k-4}\ubr{3\ldots3}_{5}\ubr{2}_{1}\ubr{111}_{3}$ &
$\ubr{5\ldots5}_{k-4}\ubr{3\ldots3}_{4}\ubr{2\ldots2}_{4}$ (\ref{eq:root-type-E48-1}\textsuperscript{$\pm$}) \\ &
$\ubr{5\ldots5}_{k-5}\ubr{444}_{3}\ubr{33}_{2}\ubr{22}_{2}\ubr{111}_{3}$ &
$\ubr{5\ldots5}_{k-5}\ubr{444}_{3}\ubr{3}_{1}\ubr{2\ldots2}_{5}$ \\ &
$\ubr{5\ldots5}_{k-5}\ubr{44}_{2}\ubr{3\ldots3}_{4}\ubr{22}_{2}\ubr{1}_{1}$ &
$\ubr{5\ldots5}_{k-6}\ubr{4\ldots4}_{6}\ubr{1\ldots1}_{6}$ (\ref{eq:root-type-sym}) \\ &
$\ubr{5\ldots5}_{k-6}\ubr{4\ldots4}_{5}\ubr{3}_{1}\ubr{222}_{3}\ubr{1}_{1}$ &
$\ubr{5\ldots5}_{k-6}\ubr{4\ldots4}_{4}\ubr{3\ldots3}_{4}\ubr{11}_{2}$ \\ &
$\ubr{5\ldots5}_{k-6}\ubr{444}_{3}\ubr{3\ldots3}_{6}$ (\ref{eq:root-type-E69-1}\textsuperscript{$-$},\ref{eq:root-type-E69-2}\textsuperscript{$+$}) &
$\ubr{5\ldots5}_{k-7}\ubr{4\ldots4}_{6}\ubr{333}_{3}\ubr{2}_{1}$ \\ &
$\ubr{5\ldots5}_{k-9}\ubr{4\ldots4}_{11}\ubr{1}_{1}$ (\ref{eq:root-type-three}') \\
\bottomrule
\end{tabular}
\caption{$W(\rA_{n-1})$-orbits of real roots of degrees $1$ to $5$, presented in the $x_i$. 
$0$'s are omitted. Some orbits are marked by the families of roots, see below.}
\label{table:real-roots}
\end{table}
The completeness of this table is justified by the following lemma. 

\begin{lemma} \label{lemma:bound on the minimal support of a roots orbit}
If $x\in\mathbb{Z}\rJ_{\infty,\infty}$ has degree $d\geqslant1$, all its entries are non-negative, $q(x)>0$ and $x=\dec(x)$, then it is a root lattice element of the natural $\rJ_{2d-1,4d-2}$ subsystem.
\end{lemma}
\begin{proof}
There are minimal $k,n$ such that 
$x$ is a root lattice element of a natural $\rJ_{k,n}$ subsystem. Write $x$ in terms of the simple roots for this 
subsystem.
Denote by $m_i$ the coefficient of $\alpha_i$ in $x$, so that $x = m_\beta\beta+m_1\alpha_1+\ldots+m_{n-1}\alpha_{n-1}$. 
Then $m_k = (VDUx)_{k+1} = x_{k+1}+\ldots+x_n$, which by \cref{lemma:w reduces degree} is at 
most $2d-1$.
It follows that $m_1,\ldots,m_k$ is a strictly increasing sequence, while $m_k,\ldots,m_{n-1}$ is strictly decreasing. Thus $k,n-k\leqslant 2d-1$, so $n\leqslant4d-2$.
\end{proof}

The above lemma gives a tool to enumerate the real roots of degree $\leqslant d$ as follows:
\begin{enumerate}
\item enumerate all length $4d-2$ decreasing sequences of numbers from $\{0,1,\ldots,d\}$ such that the sum of entries is divisible by $2d-1$;
\item for each such sequence check whether $q$ evaluates to $2$;
\item if it does, perform the procedure of \cref{theorem:positive real roots of degree d} to establish whether this sequence correspond to a real root.
\end{enumerate}
\begin{remark}
Since $x=\dec(x)$, the sequences $m_1,\ldots,m_k$ and $m_k,\ldots,m_{n-1}$ are convex in the sense that $2m_i\leqslant m_{i-1}+m_{i+1}$ for $i=2,\ldots,k-1$ and $k+1,\ldots,n-2$.
\end{remark}
\begin{remark}
Experimental evidence suggests that in fact such $x$ is an element of a natural $\rJ_{k,n}$ subsystem for some $n\leqslant 2d+2$ and some $k$. The roots $\gamma_d$ and $\delta_d$ (see below) display the extreme cases with $n=2d+2$ and $k=3$ and $k=d+1$ respectively.
\end{remark}

The rest of \cref{sec:enumeration of roots} is devoted to the description of particular series of roots. We will show that in terms of the $e_i$'s even the structure of finite type root systems is more transparent.

We start with marking two distinguished families of roots, one in each degree $d\geqslant2$:
\begin{align*}
\gamma_d & = \ubr{d\ldots d}_{k-3}\ubr{d-1}_{1}\ubr{1\ldots1}_{2d+1}, \label{eq:root-type-three}\tag{D} \\
\delta_d & = \ubr{d\ldots d}_{k-d-1}\ubr{d-1\ldots d-1}_{d+1}\ubr{1\ldots1}_{d+1}, \label{eq:root-type-sym}\tag{E}
\end{align*}
corresponding respectively to
\[ \gamma_d = \eroot{1 & d & 2d-1 & 2d-2 & \cdots & 2 & 1 \\ && d} \quad \text{and} \quad \delta_d = \eroot{1 & 2 & \cdots & d & d+1 & d & \cdots & 2 & 1 \\ &&&& d}. \]
The family of roots dual to $\gamma_d$ (see \cref{remark:dual-root-system}) are marked by (\ref{eq:root-type-three}') in \cref{table:real-roots}.

Among the roots of these series are $\gamma_2 = \delta_2$, the maximal root of the natural $\rE_6$ subsystem, and $\delta_3=\beta+\delta_{4,8}\in \rJ_{4,8}$, see \cref{sec:affine roots}. Their inners products are
\[ (\gamma_d, \gamma_{d'}) = 2 - |d-d'|, \qquad (\delta_d, \delta_{d'}) = 2 - |d-d'|, \qquad (\gamma_d, \delta_d) = d(3-d). \]
To see that they are indeed real roots note that the sum of last $n-k$ entries of $\gamma_d$ equals $2d-1$, thus $(\gamma_d,\beta) = 1$ and
\[ s_\beta(\gamma_d) = \ubr{d-1\ldots d-1}_{k-3}\ubr{d-2}_{1}\ubr{00}_{2}\ubr{1\ldots1}_{2d-1} \in W(\rA_{n-1})\gamma_{d-1}. \]
Similarly, the sum of the last $n-k$ entries of $\delta_d$ equals $d+1$, so $(\delta_d,\beta) = d-1$ and
\[ s_\beta(\delta_d) = \ubr{1\ldots1}_{k-d-1}\ubr{0\ldots0}_{d+1}\ubr{1\ldots1}_{d+1} \in W(\rA_{n-1})\beta. \]

\subsection{Root systems of finite types} \label{sec:systems of finite types}
Let us now consider the case of finite root systems $\rJ_{3,n}=\rE_n$, $n=6,7,8$. It can be seen 
in \cref{table:real-roots} that in $\rE_6$ and $\rE_7$ there are no degree $3$ roots (as in 
these cases, $\sum x_i < 3k$). 
In $\rE_6$ there is a single degree $2$ root $(1,1,1,1,1,1)^\top$, which is the maximal root 
$\gamma_2$. In $\rE_7$ there are $7$ such roots, all conjugate under $W(\rA_6)$ to 
the image of $\gamma_2$ in $\rE_7$
and of the form $\dec(x)=(1,1,1,1,1,1,0)^\top$. In $\rE_8$ there is a single $W(\rA_7)$-orbit of degree $3$ roots, of the form $\dec(x) = \gamma_3 = (2,1,1,1,1,1,1,1)^\top$.

In \cref{fig:weight diagram E6 adjoint} the positive roots of $\rE_6$ are displayed by means of the weight diagram of its adjoint representation. The weight diagram of a representation is a graph with vertices corresponding to the weights of the representation (with multiplicities). An edge labeled $i$ joins the weights $\lambda$ and $\mu$ if $\lambda-\mu=\pm\alpha_i$. The weights of the adjoint representation are the roots of the root system together with zero weights corresponding to the simple roots. To determine the root $\alpha$ corresponding to a given vertex one can find a path joining this vertex to a zero weight and going from left to right. Then $\alpha = \sum \alpha_i$, where sum is taken over all labels $i$ occuring in this path. For more details concerning weight diagrams see~\cite{PSV}.
\begin{figure}
\includegraphics{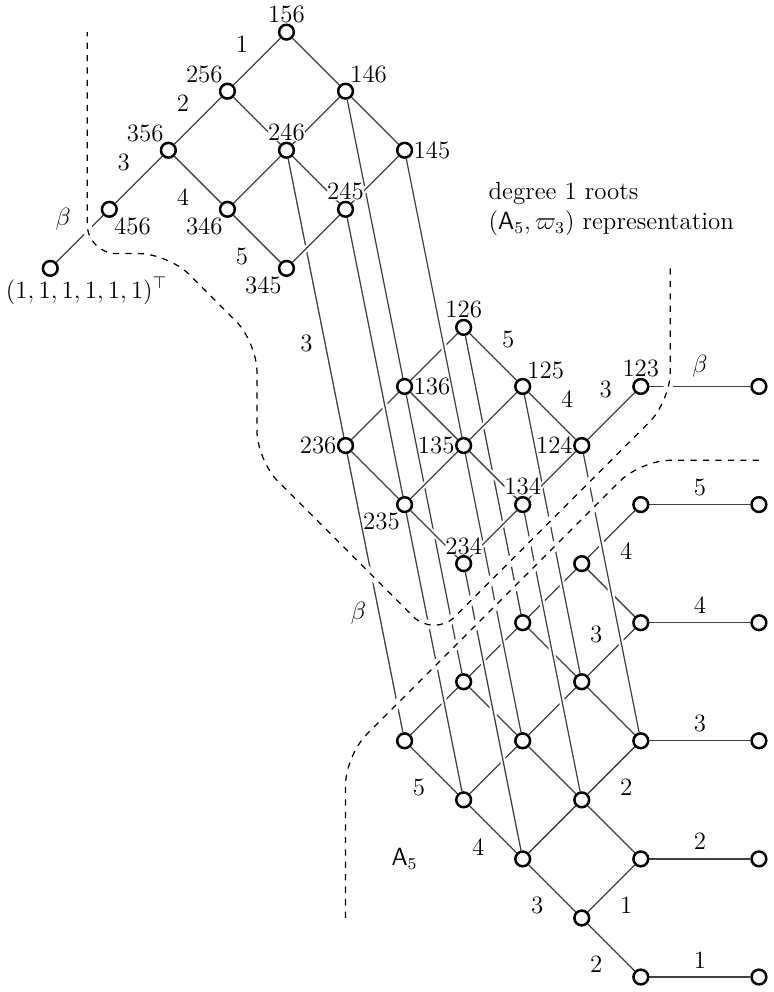}
\caption{Weight diagram of the adjoint representation of $\rE_6$. Only the part corresponding to the 
positive roots and to zero weights is shown. The rightmost vertices correspond to the zero weights 
numbered by the simple roots, and all other vertices correspond to positive roots of 
$\rE_6$. 
In any square (parallelogram) the labels on opposide sides coincide and thus some of them are 
omitted.
%The labels on the opposide sides of squares coincide, and thus some are omitted. 
The vertices of degree $1$ roots are labeled by $abc$, meaning that the corresponding roots is $e_a+e_b+e_c$. 
The leftmost vertex corresponds to $\gamma_2=e_1+e_2+\dots+e_6$.}
\label{fig:weight diagram E6 adjoint}
\end{figure}

Another instance where the $e_i$ basis reveals more symmetry is the expression for the fundamental 
weights. Recall that by definition fundamental weights (of a simply-laced root system) form the basis 
dual to the 
basis of the fundamental simple roots. The expansion of the fundamental weights in terms of the 
simple roots can be obtained by taking the columns of the matrix $A^{-1}$, the inverse 
of the Cartan matrix. Thus the expressions in terms of the $e_i$ basis can be calculated as the columns 
of $CA^{-1}$.

In case $k=2$ (so that $\rJ_{k,n}=\rD_n$) this gives (after the renumbering of the simple
roots, see \cref{example:dn}) the standard description~\cite[Ch.~VI, \S4, no.~8(VI)]{Bou1}
\begin{align*}
& \varpi_\beta = \frac{1}{2}(1,\ldots,1)^\top, \\
& \varpi_1 = \frac{1}{2}(-1,1,\ldots,1)^\top, \\
& \varpi_i = e_{i+1}+\ldots+e_n \quad \text{for} \quad 2\leqslant i\leqslant n-1.
\end{align*}
The fundamental weights for $\rE_6,\rE_7$ and $\rE_8$ are listed in \cref{table:fundamental weights E6,table:fundamental weights E7,table:fundamental weights E8}.

Similarly, the sum of all positive roots, which equals twice the sum of the fundamental weights, is
\begin{align*}
\rD_n &\colon \quad (0,1,\ldots,n-1)^\top, \\
\rE_6 &\colon \quad (3,4,5,6,7,8)^\top, \\
\rE_7 &\colon \quad (15,17,19,21,23,25,27)^\top, \\
\rE_8 &\colon \quad (22,23,24,25,26,27,28,29)^\top.
\end{align*}

\begin{table}
\begin{tabular}{l@{\qquad}c@{\qquad}c} \toprule
$\varpi_\beta$ & $\eroot{1&2&3&2&1\\&&2}$ & $(1,1,1,1,1,1)^\top$ \\
$\varpi_1$ & $\dfrac{1}{3}\left( \eroot{4&5&6&4&2\\&&3} \right)$ & $\frac{1}{3}(-1,2,2,2,2,2)^\top$ \\
$\varpi_2$ & $\dfrac{1}{3}\left( \eroot{5&10&12&8&4\\&&6} \right)$ & $\frac{1}{3}(1,1,4,4,4,4)^\top$ \\
$\varpi_3$ & $\eroot{2&4&6&4&2\\&&3}$ & $(1,1,1,2,2,2)^\top$ \\
$\varpi_4$ & $\dfrac{1}{3}\left( \eroot{4&8&12&10&5\\&&6} \right)$ & $\frac{1}{3}(2,2,2,2,5,5)^\top$ \\
$\varpi_5$ & $\dfrac{1}{3}\left( \eroot{2&4&6&5&4\\&&3} \right)$ & $\frac{1}{3}(1,1,1,1,1,4)^\top$ \\\bottomrule
\end{tabular}
\caption{Fundamental weights of $\rE_6$.}
\label{table:fundamental weights E6}
\end{table}
\begin{table}
\begin{tabular}{l@{\qquad}c@{\qquad}c} \toprule
$\varpi_\beta$ & $\dfrac{1}{2}\left( \eroot{4&8&12&9&6&3\\&&7} \right)$ & $\frac{1}{2}(3,3,3,3,3,3,3)^\top$ \\
$\varpi_1$ & $\eroot{2&3&4&3&2&1\\&&2}$ & $(0,1,1,1,1,1,1)^\top$ \\
$\varpi_2$ & $\eroot{3&6&8&6&4&2\\&&4}$ & $(1,1,2,2,2,2,2)^\top$ \\
$\varpi_3$ & $\eroot{4&8&12&9&6&3\\&&6}$ & $(2,2,2,3,3,3,3)^\top$ \\
$\varpi_4$ & $\dfrac{1}{2}\left( \eroot{6&12&18&15&10&5\\&&9} \right)$ & $\frac{1}{2}(3,3,3,3,5,5,5)^\top$ \\
$\varpi_5$ & $\eroot{2&4&6&5&4&2\\&&3}$ & $(1,1,1,1,1,2,2)^\top$ \\
$\varpi_6$ & $\dfrac{1}{2}\left( \eroot{2&4&6&5&4&3\\&&3} \right)$ & $\frac{1}{2}(1,1,1,1,1,1,3)^\top$ \\\bottomrule
\end{tabular}
\caption{Fundamental weights of $\rE_7$.}
\label{table:fundamental weights E7}
\end{table}
\begin{table}
\begin{tabular}{l@{\qquad}c@{\qquad}c} \toprule
$\varpi_\beta$ & $\eroot{5&10&15&12&9&6&3\\&&8}$ & $(3,3,3,3,3,3,3,3)^\top$ \\
$\varpi_1$ & $\eroot{4&7&10&8&6&4&2\\&&5}$ & $(1,2,2,2,2,2,2,2)^\top$ \\
$\varpi_2$ & $\eroot{7&14&20&16&12&8&4\\&&10}$ & $(3,3,4,4,4,4,4,4)^\top$ \\
$\varpi_3$ & $\eroot{10&20&30&24&18&12&6\\&&15}$ & $(5,5,5,6,6,6,6,6)^\top$ \\
$\varpi_4$ & $\eroot{8&16&24&20&15&10&5\\&&12}$ & $(4,4,4,4,5,5,5,5)^\top$ \\
$\varpi_5$ & $\eroot{6&12&18&15&12&8&4\\&&9}$ & $(3,3,3,3,3,4,4,4)^\top$ \\
$\varpi_6$ & $\eroot{4&8&12&10&8&6&3\\&&6}$ & $(2,2,2,2,2,2,3,3)^\top$ \\
$\varpi_7$ & $\eroot{2&4&6&5&4&3&2\\&&3}$ & $(1,1,1,1,1,1,1,2)^\top$ \\\bottomrule
\end{tabular}
\caption{Fundamental weights of $\rE_8$.}
\label{table:fundamental weights E8}
\end{table}

\subsection{Affine roots and roots coming from affine subsystems} \label{sec:affine roots}
Among the root systems of type $\rJ_{k,n}$ there are two affine type root systems, namely, for $(k,n)=(3,9)$ or $(4,8)$ (there is also $\rJ_{6,9}\cong \rJ_{3,9}$). In a simply-laced affine root system $\Phi$ roots come in families of the form $\alpha+m\delta$, where $\alpha$ is a root of the canonical finite type subsystem $\mathring\Phi$ and $\delta=\delta_{k,n}$ is the smallest element-wise positive vector such that $A\delta=0$ for $A$ the Cartan matrix of $\Phi$, and $m\in\mathbb{Z}$.

Denote by $J$ the $n{\times}n$-matrix consisting of $1$'s, so that the Gram matrix of the inner product with respect to the basis $e_1,\ldots,e_n$ equals $I-\frac{k-2}{k^2}J$. On the other hand, it must be equal to $C^{-\top}AC^{-1}$, so $A\delta=0$ implies $C^\top C\delta = \frac{k-2}{k^2}C^\top J C\delta$, which, in turn, means that $\delta'=C\delta$ satisfies $\delta'=\frac{k-2}{k^2}J\delta'$. If $\delta'=(x_1,\ldots,x_n)^\top$, then $x_1=\ldots=x_n=x$ and $x = \frac{k-2}{k^2}nx$. There are only three positive integer solutions to $n=\frac{k^2}{k-2}$, namely, $(k,n)=(3,9)$, $(4,8)$ or $(6,9)$. Recall that the element $x$ must also satisfy the restriction that $\sum x_i=nx$ is divisible by $k$. In cases $(k,n)=(3,9)$ and $(4,8)$ the minimal such $x$ equals $1$, and for $(k,n)=(6,9)$ one has $x=2$. Thus
\begin{align*}
& \delta'_{3,9} = (1,1,1,1,1,1,1,1,1)^\top,\\
& \delta'_{4,8} = (1,1,1,1,1,1,1,1)^\top,\\
& \delta'_{6,9} = (2,2,2,2,2,2,2,2,2)^\top,
\end{align*}
which correspond, respectively, to
\[ \delta_{k,n} = \eroot{2&4&6&5&4&3&2&1\\&&3}, \quad \eroot{1&2&3&4&3&2&1\\&&&2}, \quad \eroot{1&2&3&4&5&6&4&2\\&&&&&3}. \]
One can consider $\delta'_{k,n}$ as an element of a larger root system by means of \cref{remark:root-system-extension}, and below we will use such identifications implicitly. One also can obtain $\delta'_{6,9}$ from $\delta'_{3,9}$ by \cref{remark:dual-root-system}.

This means that if $k\geqslant3$ and $n-k\geqslant6$, there are the following families of roots in $\rJ_{k,n}$, coming from the natural $\rJ_{3,9}$ subsystem:
\begin{align*}
& \pm (e_i-e_j) + m\delta_{3,9} \quad \text{for} \quad k-3\leqslant j < i \leqslant k+5, \label{eq:root-type-E39-0}\tag{A0} \\
& \pm\ubr{1\ldots1}_{k} + m\delta_{3,9} = \ubr{3m\pm1\ldots3m\pm1}_{k-3}\ubr{m\pm1\ldots m\pm1}_{3}\ubr{m\ldots m}_{6}, \label{eq:root-type-E39-1}\tag{A1} \\
& \pm\ubr{2\ldots2}_{k-3}\ubr{1\ldots1}_{6} + m\delta_{3,9} = \ubr{3m\pm2\ldots3m\pm2}_{k-3}\ubr{m\pm1\ldots m\pm1}_{6}\ubr{m\ldots m}_{3}, \label{eq:root-type-E39-2}\tag{A2} \\
& \pm\ubr{3\ldots3}_{k-3}\ubr{2}_{1}\ubr{1\ldots1}_{7} + m\delta_{3,9} = \ubr{3m\pm3\ldots3m\pm3}_{k-3}\ubr{m\pm2}_{1}\ubr{m\pm1\ldots m\pm1}_{7}\ubr{m}_{1}. \label{eq:root-type-E39-3}\tag{A3}
\end{align*}
In \cref{table:real-roots} the orbits that contain the roots from one of the series (\ref{eq:root-type-E39-0}--\ref{eq:root-type-E39-3}) are marked by (Xi\textsuperscript{$\pm$}), with the sign corresponding to the choice of the signs in the formula.

Note that the roots of (\ref{eq:root-type-E39-1}\textsuperscript{$\pm$}) and (\ref{eq:root-type-E39-2}\textsuperscript{$\mp$}) 
families represent the same $W(\rA_{n-1})$-orbits, but with different numbering. Namely, the root of (\ref{eq:root-type-E39-1}) family with a given value of $m=m_1$ and $+$ as the sign and the root of (\ref{eq:root-type-E39-2}) family with $m=m_1+1$ and $-$ as the sign are obtained from one another by a permutation of the last $9$ non-zero entries. The root of the form (\ref{eq:root-type-E39-1}) with $m=m_1$ and $-$ as the sign is in the same $W(\rA_{n-1})$-orbit as the root of the form (\ref{eq:root-type-E39-2}) with $m=m_1-1$ and $+$ for the sign.

The same holds for (\ref{eq:root-type-E39-3}\textsuperscript{$+$}) and (\ref{eq:root-type-E39-3}\textsuperscript{$-$}).

Similarly if $k\geqslant6$ and $n-k\geqslant3$, there are roots coming from the natural $\rJ_{6,9}$ subsystem:
\begin{align*}
& \pm(e_i-e_j)+m\delta_{6,9} \quad \text{for} \quad k-5\leqslant j < i \leqslant k+3, \label{eq:root-type-E69-0}\tag{B0} \\
& \pm\ubr{1\ldots1}_{k} + m\delta_{6,9} = \ubr{3m\pm1\ldots3m\pm1}_{k-6}\ubr{2m\pm1\ldots 2m\pm1}_{6}\ubr{2m\ldots 2m}_{3}, \label{eq:root-type-E69-1}\tag{B1} \\
& \pm\ubr{2\ldots2}_{k-3}\ubr{1\ldots1}_{6} + m\delta_{6,9} = \ubr{3m\pm2\ldots3m\pm2}_{k-6}\ubr{2m\pm2\ldots 2m\pm2}_{3}\ubr{2m\pm1\ldots 2m\pm1}_{6}, \label{eq:root-type-E69-2}\tag{B2} \\
& \pm\ubr{3\ldots3}_{k-5}\ubr{2\ldots2}_{7}\ubr{1}_{1} + m\delta_{6,9} = \ubr{3m\pm3\ldots3m\pm3}_{k-6}\ubr{2m\pm3}_{1}\ubr{2m\pm2\ldots2m\pm2}_{7}\ubr{2m\pm1}_{1}. \label{eq:root-type-E69-3}\tag{B3}
\end{align*}
Again, the roots of (\ref{eq:root-type-E69-1}\textsuperscript{$\pm$} and \ref{eq:root-type-E69-2}\textsuperscript{$\mp$}) and (\ref{eq:root-type-E69-3}\textsuperscript{$+$} and \ref{eq:root-type-E69-3}\textsuperscript{$-$}) represent the same $W(\rA_{n-1})$-orbits, but with different numbering.

If $k,n-k\geqslant4$, there are roots coming from the natural $\rJ_{4,8}$ subsystem:
\begin{align*}
& \pm(e_i-e_j)+m\delta_{4,8} \quad \text{for} \quad k-4\leqslant j < i \leqslant k+3, \label{eq:root-type-E48-0}\tag{C0} \\
& \pm\ubr{1\ldots1}_{k} + m\delta_{4,8} = \ubr{2m\pm1\ldots2m\pm1}_{k-4}\ubr{m\pm1\ldots m\pm1}_{4}\ubr{m\ldots m}_{4}, \label{eq:root-type-E48-1}\tag{C1} \\
& \pm\ubr{2\ldots2}_{k-3}\ubr{1\ldots1}_{6} + m\delta_{4,8} = \ubr{2m\pm2\ldots2m\pm2}_{k-4}\ubr{m\pm2}_{1}\ubr{m\pm1\ldots m\pm1}_{6}\ubr{m}_{1}. \label{eq:root-type-E48-2}\tag{C2}
\end{align*}
Here $W(\rA_{n-1})$-orbits do not depend on the signs (after a suitable renumbering), and, moreover, the orbits of (\ref{eq:root-type-E48-0}) and (\ref{eq:root-type-E48-2}) cover the same set of roots.

\subsection{Almost real roots} \label{sec:almost real roots}
Every real root of degree $\geqslant1$, when expressed in basis $e_1,\ldots,e_n$, say, $x=(x_1,\ldots,x_n)^\top$, satisfies the following three properties by \cref{lemma:q is invariant,lemma:0<x<d is preserved}:
\begin{enumerate}
\item $x\in\mathbb{Z}\Delta$,
\item $0\leqslant x_1,\ldots,x_n\leqslant d$, where $d=(x_1+\ldots+x_n)/k$,
\item $q(x) = \sum x_i^2 + \frac{2-k}{k^2}(\sum x_i)^2 = 2$.
\end{enumerate}
However, there are vectors $x$ satisfying all of the above, which are not real roots. We call such elements of the root lattice \emph{almost real roots}. They exist in degrees $\geqslant4$. 

%\revise{ Is the following correct?
%\begin{theorem}
%For $k=1,2,3$ and $n \in \mathbb{Z}_{\ge k}$, there is no almost real root in $\rJ_{k,n}$.
%\end{theorem}
%\begin{proof}
%\end{proof}
%}

Almost real roots of degrees $4$ and $5$ (in $\rJ_{k,n}$ for large enough $k,n$) are listed in 
\cref{table:almost-real-roots}. 
Note that the statement of \cref{lemma:bound on the minimal support of a roots orbit} also 
holds for almost real roots.

\begin{table}[h]
\begin{tabular}{l@{\qquad}l} \toprule
degree $4$ & $\ubr{4\ldots4}_{k-4}\ubr{333}_{3}\ubr{1\ldots1}_{7}$ \quad $\ubr{4\ldots4}_{k-6}\ubr{3\ldots3}_{7}\ubr{111}_{3}$ \\ \midrule
degree $5$ &
$\ubr{5\ldots5}_{k-3}\ubr{33}_{2}\ubr{1\ldots1}_{9}$ \quad 
$\ubr{5\ldots5}_{k-4}\ubr{44}_{2}\ubr{3}_{1}\ubr{2}_{1}\ubr{1\ldots1}_{7}$ \quad 
$\ubr{5\ldots5}_{k-5}\ubr{4\ldots4}_{4}\ubr{22}_{2}\ubr{1\ldots1}_{5}$ \\ &
$\ubr{5\ldots5}_{k-6}\ubr{4\ldots4}_{5}\ubr{33}_{2}\ubr{1\ldots1}_{4}$ \quad
$\ubr{5\ldots5}_{k-7}\ubr{4\ldots4}_{7}\ubr{3}_{1}\ubr{2}_{1}\ubr{11}_{2}$ \quad
$\ubr{5\ldots5}_{k-8}\ubr{4\ldots4}_{9}\ubr{22}_{2}$ \\
\bottomrule
\end{tabular}
\caption{$W(\rA_{n-1})$-orbits of almost real roots of degrees $4$ and $5$.}
\label{table:almost-real-roots}
\end{table}

Calculating the minimal subsystems for each orbit ot almost real roots in \cref{table:almost-real-roots} we see that degree $4$ almost real roots are present in all root systems of type $\rJ_{k,n}$ which contain $\rJ_{4,10}$ or $\rJ_{6,10}$, and that every $\rJ_{k,n}$ root system containing $\rJ_{3,11}$ or $\rJ_{8,11}$ has an almost real root of degree $5$. This implies that in every non-finite, non-affine, non-hyperbolic root system of type $\rJ_{k,n}$ there are almost real roots.

Let $x$ be a positive almost real root, and assume that $x=\dec(x)$. Then repeating the operation $x\mapsto\dec(s_\beta(x))$ lowers the height and eventually leads to an element $x'=(x_1,\ldots,x_n)$ of degree $d$ such that either some of the entries $x_i$ are negative or greater than $d$. When translated to the root basis, this means that the coefficient of some simple root $\alpha_i$ is negative.
\begin{example}
For $(k,n)=(4,10)$ set $x=(3,3,3,1,1,1,1,1,1,1)^\top$, so that in the root basis
\[ x = \eroot{1&2&3&6&5&4&3&2&1\\&&&4}. \]
Then $\dec(s_\beta(x)) = (1,1,1,1,1,1,1,1,1,-1)^\top$, which it the roots basis equals
\[ \dec(s_\beta(x)) = \eroot{1&2&3&4&3&2&1&0&-1\\&&&2}. \]
Similarly, for $(k,n)=(6,10)$ and $x=(3,3,3,3,3,3,3,1,1,1)^\top$ in terms of the root basis one has
\[ x = \eroot{1&2&3&4&5&6&3&2&1\\&&&&&4}, \]
while $\dec(s_\beta(x)) = (3,1,1,1,1,1,1,1,1,1)^\top$, which translates into
\[ \dec(s_\beta(x)) = \eroot{-1&0&1&2&3&4&3&2&1\\&&&&&2}. \]
\end{example}

To find real roots, we consider the orbits under $W(\rA_{n-1})$ (\cref{remark:W-orbits}).
The numbers of $W(\rA_{n-1})$-orbits of real roots and almost real roots are listed in \cref{table:roots orbits count in Ekn}, and the total numbers of real roots and almost real roots are listed in \cref{table:real roots count,table:almost real roots count} respectively. 

\section{Comparison with Manin's hyperbolic construction} \label{sec:hyperbolic construction}
In \cite{Manin} Manin gave a construction of the root system $\rE_8$ inside a hyperbolic lattice. His construction works as follows.

Consider a $9$-dimensional space $V$ equipped with the inner product of signature $(1,8)$. This means that there exists an orthogonal basis $f_0,f_1,\ldots,f_8$ of $V$ such that $(f_0,f_0)=1$, $(f_i,f_i)=-1$ for $i\geqslant 1$. Set $\omega = -3f_0+f_1+\ldots+f_8$ and define the lattice $L=\mathbb{Z}f_0+\ldots+\mathbb{Z}f_8$. Then the set
\[ R = \{ f\in L \mid (f,\omega)=0,\ (f,f)=-2 \} \]
is the root system of type $\rE_8$ \cite[Proposition~25.2 and Theorem~25.4]{Manin}.

This realization is related to the structure of del Pezzo surfaces. If a del Pezzo surface $V$ of degree $d$ is not isomorphic to $\mathbb{P}^1\times\mathbb{P}^1$, then its Picard group $\operatorname{Pic}(V)$ is isomorphic to the odd unimodular lattice $L = I_{1,9-d}$, in which the root system is realised.

The complete enumeration of roots of $\rE_8$ is provided by \cite[Proposition~25.5.3]{Manin}. 
It states that if $(a,b_1,\ldots,b_8)$ are the coordinates of a root with respect to the basis 
$f_0,f_1,\ldots,f_8$, then these coordinates can be obtained from the rows of the following table by a permutation of the last $8$ entries $b_1,\ldots,b_8$ and, possibly, a simultaneous change of the sign for all $9$ entries :
\[ \begin{tabular}{>{$}c<{$}>{$}c<{$}>{$}c<{$}>{$}c<{$}>{$}c<{$}>{$}c<{$}>{$}c<{$}>{$}c<{$}>{$}c<{$}}
\toprule
a & b_1 & b_2 & b_3 & b_4 & b_5 & b_6 & b_7 & b_8 \\\midrule
0 & 1 & 0 & 0 & 0 & 0 & 0 & 0 & -1 \\
1 & 1 & 1 & 1 & 0 & 0 & 0 & 0 & 0 \\
2 & 1 & 1 & 1 & 1 & 1 & 1 & 0 & 0 \\
3 & 2 & 1 & 1 & 1 & 1 & 1 & 1 & 1 \\
\bottomrule
\end{tabular}
\]

Comparing this with the content of \cref{table:real-roots} reveals that this construction coincides with our presentation of $\rE_8$-roots insise $\rJ_{4,9}$. The correspondence, from presentation in 
the $x_i$ to presentation in the $f_i$, is as follows 
\[ x=(x_1,\ldots,x_8) \rightsquigarrow (\deg(x),x_1,\ldots,x_8), \]
(extending the presentations from \cref{table:real-roots} by $0$s at the end where needed). 
For degree $0$, see \cref{remark:degree 0 roots}.
This is exactly the inclusion $\rJ_{3,8}\hookrightarrow \rJ_{4,9}$ described in \cref{remark:root-system-extension}.

Moreover, the exceptional curves on $V$ are parametrised by the following elements of $\operatorname{Pic}(V)$ (with respect to $f_0,f_1,\ldots,f_n)$:
\[
\begin{tabular}{>{$}c<{$}>{$}c<{$}>{$}c<{$}>{$}c<{$}>{$}c<{$}>{$}c<{$}>{$}c<{$}>{$}c<{$}>{$}c<{$}}
\toprule
a & b_1 & b_2 & b_3 & b_4 & b_5 & b_6 & b_7 & b_8 \\\midrule
0 & -1 & 0 & 0 & 0 & 0 & 0 & 0 & 0 \\
1 & 1 & 1 & 0 & 0 & 0 & 0 & 0 & 0 \\
2 & 1 & 1 & 1 & 1 & 1 & 0 & 0 & 0 \\
3 & 2 & 1 & 1 & 1 & 1 & 1 & 1 & 0 \\
4 & 2 & 2 & 2 & 1 & 1 & 1 & 1 & 1 \\
5 & 2 & 2 & 2 & 2 & 2 & 2 & 1 & 1 \\
6 & 3 & 2 & 2 & 2 & 2 & 2 & 2 & 2 \\
\bottomrule
\end{tabular}
\]
together with all obtained by permuting $b_1,\ldots,b_8$. The calculation of these elements in \cite[Proposition~26.1]{Manin} is done by introducing an auxiliary parameter $b_9$ and then specifying it to $1$. Note, however, that for such values of $b_i$ the vector $(a,b_1,\ldots,b_8,1)$ coincides with one of the roots of $\rJ_{4,10}$ coming from the affine subsystem $\rJ_{3,9}$, given by formulas~(\ref{eq:root-type-E39-0})--(\ref{eq:root-type-E39-3}) with $m=1$. Namely, the image $\delta$ of $\delta_{3,9}$ in $\rJ_{4,10}$ is $(3,1,1,1,1,1,1,1,1,1,1)^\top$, so
\[
\begin{tabular}{r<{$\colon$}@{\qquad}>{$}l<{$}}
\text{(\ref{eq:root-type-E39-3}\textsuperscript{$-$})} & -\ubr{3}_{1}\ubr{2}_1\ubr{1\ldots 1}_{7} + \delta = (0,-1,0,0,0,0,0,0,0,1)^\top, \\
\text{(\ref{eq:root-type-E39-2}\textsuperscript{$-$})} & -\ubr{2}_{1}\ubr{1\ldots1}_{6} + \delta = (1,0,0,0,0,0,0,1,1,1)^\top, \\
\text{(\ref{eq:root-type-E39-1}\textsuperscript{$-$})} & -\ubr{1\ldots 1}_{4} + \delta = (2,0,0,0,1,1,1,1,1,1)^\top, \\
\text{(\ref{eq:root-type-E39-0})} & \vphantom{\ubr{1}_{1}} e_2-e_8 + \delta = (3,2,1,1,1,1,1,1,0,1)^\top, \\
\text{(\ref{eq:root-type-E39-1}\textsuperscript{$+$})} & \ubr{1\ldots 1}_{4} + \delta = (4,2,2,2,1,1,1,1,1,1)^\top, \\
\text{(\ref{eq:root-type-E39-2}\textsuperscript{$+$})} & \ubr{2}_{1}\ubr{1\ldots1}_{6} + \delta = (5,2,2,2,2,2,2,1,1,1)^\top, \\
\text{(\ref{eq:root-type-E39-3}\textsuperscript{$+$})} & \ubr{3}_{1}\ubr{2}_1\ubr{1\ldots 1}_{7} + \delta = (6,3,2,2,2,2,2,2,2,1)^\top.
\end{tabular}
\]

\section{Connection with cluster algebras} \label{sec:cluster algebras}

Jensen, King and Su \cite{JKS} have given an additive categorification of the cluster algebra structure 
on the coordinate ring $\mathbb{C}[\Gr_{k,n}]$ of the Grassmannian of $k$-subspaces in $n$-space, 
by considering the category $\mathrm{CM}(B_{k,n})$ of Cohen-Macaulay modules over a 
quotient $B_{k,n}$ of the preprojective algebra of type $\widetilde{\rA}$.

Jensen, King and Su pointed out in \cite[Section 8]{JKS} that in the finite type cases,  indecomposable modules corresponds to real roots in the associated root system and that the number of indecomposable rank $d$ modules is $d$ times the number of real roots of degree $d$. 
They observe that this evidence suggests that rigid indecomposable modules correspond to roots (as classes in the Grothendieck group) 
and that for every real root of degree $d$ there are $d$ rigid indecomposable objects of rank $d$. They showed that the Grothendieck group of ${\rm CM}(B_{k,n})$ can be identified with the root lattice $\Lambda(\rJ_{k,n})$ and with the sublattice $\mathbb{Z}^n(k) \subset \mathbb{Z}^n$ spanned by the  $GL_n(\mathbb{C})$ weights of the homogeneous functions in $\CC[\Gr(k,n)]$. Thus their ``root conjecture'' means that the weights of cluster variables are roots of $\rJ_{k,n}$. 

%King conjectured that the number of rigid indecomposable rank $d$ modules in $\mathrm{CM}(B_{k,n})$ 
%is $d$ times the number of real roots of degree $d$ in $\rJ_{k,n}$, this is true in the finite types. 
%Furthermore, Jensen, King and Su conjectured that 
%every rigid indecomposable module in $\mathrm{CM}(B_{k,n})$ is associated to a (real or imaginary) 
%root in $\rJ_{k,n}$.

Every $B_{k,n}$-module of rank $1$ can be characterized by a $k$-element subset of $\{1,\ldots,n\}$, see \cite[Definition~5.1 and Proposition~5.2]{JKS}. 
These in turn correspond to real roots in degree $1$.

The rank $1$ modules can be viewed as building blocks for the category as every module in 
${\rm CM}(B_{k,n})$ has a filtration with factors which are rank $1$ modules, as pointed out in a private communication by 
A.~King and M.~Pressland. 
If $M$ is an arbitrary module in ${\rm CM}(B_{k,n})$, 
one can consider homormophisms $L_I\hookrightarrow M$ such that the quotient $M/L_I$ is also in ${\rm CM}(B_{k,n})$. Such 
homomorphisms always exist and allow to reduce the rank of $M$. 
Such a filtration is not unique in general.
Let $M$ be a rank $n$ module in ${\rm CM}(B_{k,n})$ with factors $L_{I_1},\dots, L_{I_d}$ in its filtration, where $L_{I_d}$ is a submodule of $M$. We write
\[ P_M=\thfrac{I_1}{\stackrel{\vdots {\phantom{a}}}{\phantom{.} }}{I_d} \quad \text{or} \quad P_M=I_1| \cdots |I_d, \]
and $P_M$ is called the \emph{profile} of $M$. The number $d$ is called the {\sl rank} of the module $M$. 

For every module $M$ with a profile $P_M$ of $d$ rows, 
one associates the element $\varphi(M)=\varphi(P_M):= (x_1,\ldots,x_n)^\top$ in $\mathbb{Z}\Delta$ 
where $x_i$ is the number of occurrences of $i$ in the profile of $M$. 

Indeed, since each of these $d$ rows has size $k$, the total number of entries is $x_1+\ldots+x_n=kd$. We have $0\leqslant x_i\leqslant d$ for each $i \in \{1,\ldots, n\}$. 

Conversely, for any element $x=(x_1,\ldots,x_n)^\top\in\mathbb{Z}\Delta$ with 
$0\leqslant x_i\leqslant d$ for all $i$, one can construct a profile mapping to $x$. To do this, 
take the sequence 
\[ 
a = \big(\underbrace{1,\ldots,1}_{x_1},\underbrace{2,\ldots,2}_{x_2},\ldots,\underbrace{n,\ldots,n}_{x_n}\big)
\] 
of length $kd$ and set
\[ I_i = \{a_{k-i+1},a_{k-i+1+d},\ldots,a_{k-i+1+(k-1)d}\}. \]
Then define $P_x:=I_1|\cdots|I_d$. This is a profile with $\varphi(P_x)=x$.
For example, for the root $x=(2,1,1,1,1,1,1,1)$ in $\rJ_{3,8}$, this produces 
\[ P_x = \thfrac{258}{147}{136}. \] 

Now given a profile $P$ with $d$ rows, we order the entries increasingly in each row and write this as 
$P=(P_{ij})$, $1\le i\le d$, $1\le j\le k$. So $(P_{ij})_j$ is the $i$th row of the profile and 
$P_{ij} < P_{ij'}$ for $j<j'$. The profile $P$ is called {\sl weakly column decreasing} if 
for every $j\in[k]$ and for every $i\in[d-1]$, we have $P_{i,j} \ge P_{i+1,j}$. If $P$ is weakly column 
decreasing and in addition, 
we have $P_{d,j} \ge P_{1, j-1}$ for all $j\in [2,k]$, we say that $P$ is \emph{canonical}.

The profile $P_x$ corresponding to $x \in \mathbb{Z}\Delta$ constructed above is a canonical profile. In \cite[Theorem 5.7]{BBGL}, it is shown that the profile of any rigid indecomposable module of rank $3$ such that $\varphi(M)$ is a real root is a cyclic permutation of a canonical profile. For example, the profile
\[ P=\thfrac{258}{147}{136} \]
is a canonical profile of rank $3$ and $\varphi(P)$ is a real root in $\rJ_{3,8}$. The cyclic permutations of $P$ are
\[ \thfrac{258}{147}{136},\quad \thfrac{147}{136}{258}\quad \text{and}\quad \thfrac{136}{258}{147}. \]
The modules with these profiles are all rigid indecomposable.
We note that it is conjectured that whenever $M$ in ${\rm CM}(B_{k,n})$ is rigid indecomposable 
and $\varphi(M)$ is a real root in $\rJ_{k,n}$, then the profile $P_M$ is a cyclic permutation of a canonical profile, \cite[Conjecture~5.8]{BBGL}.

The results about real roots in $\rJ_{k,n}$ in this paper are thus expected to help with the 
characterization of rigid indecomposable modules in ${\rm CM}(B_{k,n})$ 
corresponding to real roots.

\begin{table}\footnotesize
\begin{tabular}{l@{\qquad}l@{\qquad}*{12}{c}}\toprule
& \qquad degree & 1 & 2 & 3 & 4 & 5 & 6 & 7 & 8 & 9 & 10 & 11 \\\midrule
\multirow{2}{*}{$(k,n)=(\infty,\infty)$} & real roots & 1 & 1 & 3 & 8 & 17 & 37 & 72 & 139 & 253 & 439 & 722 \\
& almost r. r. & 0 & 0 & 0 & 2 & 6 & 20 & 65 & 153 & 390 & 878 & 1888 \\\midrule
\multirow{2}{*}{$(k,n)=(3,\infty)$}& real roots & 1 & 1 & 1 & 2 & 3 & 5 & 7 & 13 & 17 & 28 & 37 \\
& almost r. r. & 0 & 0 & 0 & 0 & 1 & 1 & 4 & 7 & 16 & 27 & 52 \\\midrule
\multirow{2}{*}{$(k,n)=(4,\infty)$} & real roots & 1 & 1 & 2 & 4 & 8 & 15 & 26 & 44 & 76 & 115 & 183 \\
& almost r. r. & 0 & 0 & 0 & 1 & 2 & 5 & 15 & 31 & 64 & 131 & 250 \\\midrule
\multirow{2}{*}{$(k,n)=(5,\infty)$} & real roots & 1 & 1 & 3 & 6 & 11 & 24 & 45 & 81 & 143 & 236 & 372 \\
& almost r. r. & 0 & 0 & 0 & 1 & 3 & 9 & 26 & 53 & 133 & 266 & 529 \\\midrule
\multirow{2}{*}{$(k,n)=(3,10)$} & real roots & 1 & 1 & 1 & 2 & 2 & 2 & 3 & 5 & 5 & 7 & 9 \\
&almost r. r. & 0 & 0 & 0 & 0 & 0 & 0 & 0 & 0 & 0 & 0 & 0 \\\midrule
\multirow{2}{*}{$(k,n)=(3,11)$} & real roots & 1 & 1 & 1 & 2 & 2 & 4 & 4 & 8 & 10 & 14 & 18 \\
& almost r. r. & 0 & 0 & 0 & 0 & 1 & 0 & 1 & 0 & 2 & 1 & 3 \\\midrule
\multirow{2}{*}{$(k,n)=(3,12)$} & real roots & 1 & 1 & 1 & 2 & 3 & 4 & 6 & 10 & 13 & 20 & 27 \\
& almost r. r. & 0 & 0 & 0 & 0 & 1 & 1 & 2 & 2 & 5 & 5 & 9 \\\midrule
\multirow{2}{*}{$(k,n)=(4,9)$} & real roots & 1 & 1 & 2 & 2 & 3 & 5 & 7 & 9 & 14 & 17 & 22 \\
& almost r. r. & 0 & 0 & 0 & 0 & 0 & 0 & 0 & 0 & 0 & 0 & 0 \\\midrule
\multirow{2}{*}{$(k,n)=(4,10)$} & real roots & 1 & 1 & 2 & 3 & 6 & 8 & 15 & 20 & 34 & 44 & 70 \\
& almost r. r. & 0 & 0 & 0 & 1 & 0 & 1 & 1 & 3 & 1 & 8 & 4 \\\midrule
\multirow{2}{*}{$(k,n)=(4,11)$} & real roots & 1 & 1 & 2 & 4 & 7 & 12 & 20 & 31 & 52 & 74 & 117 \\
& almost r. r. & 0 & 0 & 0 & 1 & 1 & 2 & 4 & 8 & 10 & 24 & 32 \\\midrule
\multirow{2}{*}{$(k,n)=(5,10)$} & real roots & 1 & 1 & 3 & 4 & 6 & 12 & 21 & 31 & 52 & 76 & 110 \\
& almost r. r. & 0 & 0 & 0 & 0 & 0 & 0 & 0 & 0 & 2 & 2 & 2 \\
\bottomrule
\end{tabular}
\caption{Number of $W(\rA_{n-1})$-orbits of real roots and almost real roots in $\rJ_{k,n}$.}
\label{table:roots orbits count in Ekn}
\end{table}

\begin{table}\footnotesize
\begin{tabular}{l@{\qquad}l@{\qquad}*{7}{c}}\toprule
&& \multicolumn{7}{c}{$d$} \\\cmidrule{3-9}
$k$ & $n$ & 1 & 2 & 3 & 4 & 5 & 6 & 7 \\\midrule
\multirow{10}{*}{3}
& 6 & 20 & 1 & 0 & 0 & 0 & 0 & 0\\
& 7 & 35 & 7 & 0 & 0 & 0 & 0 & 0\\
& 8 & 56 & 28 & 8 & 0 & 0 & 0 & 0\\
& 9 & 84 & 84 & 72 & 84 & 84 & 72 & 84\\
& 10 & 120 & 210 & 360 & 850 & 1680 & 3870 & 7560\\
& 11 & 165 & 462 & 1320 & 4730 & 13860 & 42240 & 106260\\
& 12 & 220 & 924 & 3960 & 19140 & 73932 & 267300 & 802164\\
& 13 & 286 & 1716 & 10296 & 62920 & 300456 & 1235520 & 4241952\\
& 14 & 364 & 3003 & 24024 & 178178 & 1010100 & 4618628 & 17669652\\
& 15 & 455 & 5005 & 51480 & 450450 & 2948400 & 14774970 & 61861800\\\hline
\multirow{8}{*}{4}
& 8 & 70 & 56 & 70 & 56 & 70 & 56 & 70\\
& 9 & 126 & 252 & 702 & 1764 & 4914 & 9828 & 24390\\
& 10 & 210 & 840 & 3870 & 15960 & 55020 & 159480 & 419460\\
& 11 & 330 & 2310 & 15510 & 87890 & 355740 & 1276110 & 3626040\\
& 12 & 495 & 5544 & 50490 & 361680 & 1683990 & 6965640 & 21521610\\
& 13 & 715 & 12012 & 141570 & 1221792 & 6456606 & 29673072 & 99664422\\
& 14 & 1001 & 24024 & 354354 & 3571568 & 21191352 & 105921816 & 385453068\\
& 15 & 1365 & 45045 & 810810 & 9339330 & 61637940 & 330720600 & 1297836540\\\hline
\multirow{6}{*}{5}
& 10 & 252 & 1260 & 7020 & 30492 & 117180 & 330120 & 950220\\
& 11 & 462 & 4620 & 39930 & 243012 & 1113420 & 3903240 & 12134760\\
& 12 & 792 & 13860 & 166320 & 1292412 & 6763680 & 27642780 & 92038320\\
& 13 & 1287 & 36036 & 563706 & 5305872 & 31081050 & 142573860 & 506859210\\
& 14 & 2002 & 84084 & 1645644 & 18138120 & 117466440 & 590545956 & 2235937704\\
& 15 & 3003 & 180180 & 4285710 & 54029976 & 383439420 & 2079637560 & 8363775420\\\hline
\multirow{4}{*}{6}
& 12 & 924 & 18480 & 239580 & 1899744 & 10308144 & 41888880 & 143037840\\
& 13 & 1716 & 60060 & 1055340 & 10249096 & 63075012 & 288004860 & 1057150380\\
& 14 & 3003 & 168168 & 3777774 & 43259216 & 295387092 & 1482785304 & 5793796008\\
& 15 & 5005 & 420420 & 11621610 & 152912760 & 1143127440 & 6211345140 & 25687061400\\\hline
\multirow{2}{*}{7}
& 14 & 3432 & 210210 & 4924920 & 57028972 & 396203808 & 1987088532 & 7851283440\\
& 15 & 6435 & 630630 & 18648630 & 249909660 & 1917115200 & 10417968990 & 43770406680\\
\bottomrule
\end{tabular}
\caption{Number of real roots of a given degree in $\rJ_{k,n}$.}
\label{table:real roots count}
\end{table}

\begin{table}\footnotesize
\begin{tabular}{l@{\qquad}l@{\qquad}*{4}{c}}\toprule
&& \multicolumn{4}{c}{$d$} \\\cmidrule{3-6}
$k$ & $n$ & 4 & 5 & 6 & 7 \\\midrule
\multirow{6}{*}{3}
& 10 & 0 & 0 & 0 & 0\\
& 11 & 0 & 55 & 0 & 462\\
& 12 & 0 & 660 & 1320 & 13464\\
& 13 & 0 & 4290 & 17160 & 148434\\
& 14 & 0 & 20020 & 120120 & 1021020\\
& 15 & 0 & 75075 & 600600 & 5225220\\\hline
\multirow{7}{*}{4}
& 9 & 0 & 0 & 0 & 0\\
& 10 & 120 & 0 & 1260 & 840\\
& 11 & 1320 & 3960 & 41580 & 138600\\
& 12 & 7920 & 48180 & 445500 & 1953864\\
& 13 & 34320 & 317460 & 2925780 & 15038452\\
& 14 & 120120 & 1501500 & 14294280 & 82496414\\
& 15 & 360360 & 5705700 & 56936880 & 360751755\\\hline
\multirow{6}{*}{5}
& 10 & 0 & 0 & 0 & 0\\
& 11 & 1320 & 6930 & 62832 & 274890\\
& 12 & 15840 & 130680 & 1197504 & 5959800\\
& 13 & 102960 & 1162590 & 11052756 & 60911136\\
& 14 & 480480 & 6906900 & 68757689 & 412876464\\
& 15 & 1801800 & 31531500 & 329924595 & 2135223090\\\hline
\multirow{4}{*}{6}
& 12 & 15840 & 166320 & 1507968 & 8149680\\
& 13 & 137280 & 1930500 & 18666648 & 110630520\\
& 14 & 840840 & 14434420 & 148420272 & 943518576\\
& 15 & 3963960 & 80029950 & 871065195 & 5889723840\\\hline
\multirow{2}{*}{7}
& 14 & 960960 & 18018000 & 187675488 & 1224431208\\
& 15 & 5405400 & 121696575 & 1356755400 & 9474134670\\
\end{tabular}
\caption{Number of almost real roots of a given degree in $\rJ_{k,n}$.}
\label{table:almost real roots count}
\end{table}

\clearpage

%\bibliographystyle{amsalpha}
%\bibliography{real_roots_in_the_root_system_ekn.bib}

\end{document}